\documentclass[11pt]{amsart}
\usepackage{color}
\usepackage{palatino}
\usepackage{hyperref}
\usepackage{url}
\usepackage[nice]{nicefrac}
\renewcommand{\baselinestretch}{1.0}
\usepackage{amscd}
\usepackage{amsthm}
\usepackage{amsxtra}
\usepackage{amsfonts,amssymb,amsmath,mathrsfs,graphicx}
\usepackage{float}
\usepackage{epsfig}
\usepackage{graphicx}
\usepackage{verbatim}
\renewcommand{\baselinestretch}{\baselinestretch}
\renewcommand{\baselinestretch}{1.2}
\hypersetup{
bookmarks=true, 
unicode=false, 
pdftoolbar=true, 
pdfmenubar=true, 
pdffitwindow=false, 
pdftitle={Extensions of Calabi's correspondence between minimal surfaces and maximal surfaces}, 
pdfauthor={Hojoo Lee}, 
pdfnewwindow=true, 
colorlinks=true, 
linkcolor=black, 
citecolor=blue, 
filecolor=magenta, 
urlcolor=blue 
}
\newtheorem{theorem}{Theorem}[section]
\newtheorem{lemma}[theorem]{Lemma}
\newtheorem{corollary}[theorem]{Corollary}
\theoremstyle{definition}
\newtheorem{definition}[theorem]{Definition}
\newtheorem{example}[theorem]{Example}

\newtheorem{prop}[theorem]{Proposition}

\theoremstyle{remark}
\newtheorem{remark}[theorem]{Remark}
\numberwithin{equation}{section}

\begin{document}
\title[Extensions of Calabi's correspondence]{Extensions of Calabi's correspondence between minimal surfaces and maximal surfaces}
\author{Hojoo Lee}
\keywords{Poincar\'{e} lemma, mean curvature, base curvature, bundle curvature}
\maketitle
\textbf{\textsc{Abstract.}} In 1970, Calabi introduced the duality between minimal graphs in Euclidean space ${\mathbb{R}}^{3}$ and maximal graphs in  Lorentz space ${\mathbb{L}}^{3}$. The main goal of this survey is to illustrate geometric applications of the Poincar\'{e} Lemma to constant mean curvature equations. We construct two extensions of Calabi's correspondence. This work implements Shiffman's 1956 proposal. We discover the first twin correspondence between CMC $H$ graphs in $\mathbb{E}^{3}(\kappa, \tau)$ with base curvature $\kappa$ and bundle curvature $\tau$ and spacelike CMC $\tau$ graphs in $\mathbb{L}^{3}(\kappa, H)$ with base curvature $\kappa$ and bundle curvature $H$. For instance, the twin correspondence induces the duality between CMC $H$ graphs in Riemannian Heisenberg group ${\mathrm{Nil}}^{3}(\tau)=\mathbb{E}^{3}\left(0, \tau\right)$ and spacelike CMC $\tau$ graphs in Lorentzian Heisenberg group ${\mathrm{Nil}}^{3}_{1}(H)=\mathbb{L}^{3}\left(0, H\right)$. We employ Chern's Theorem and the twin correspondence to establish that there exists no entire spacelike graph of zero mean curvature in Lorentzian Heisenberg group with non-zero bundle curvature. We discover the second twin correspondence between two dimensional minimal graphs having positive area angle functions in Euclidean space ${\mathbb{R}}^{n+2}$ and two dimensional maximal graphs having the same positive area angle functions in pseudo-Euclidean space ${\mathbb{R}}^{n+2}_{n}$ with signature $(+,+,-,\cdots,-)$. The twin correspondence induces a duality for special Lagrangian graphs in ${\mathbb{R}}^{4}$ and ${\mathbb{R}}^{4}_{2}$ with prescribed Lagrangian angles.

\bigskip

{\small \textbf{\textsc{Acknowledgement.}}} {\small I would like to deeply thank my advisor Jaigyoung
Choe for introducing me to the theory of minimal submanifolds, one of the most beautiful as well as useful branches of Mathematics.} {\small I also would like to thank Tom Wan for sending me the paper \cite{HTTW95} and Marc Soret
for introducing me the recent preprint \cite{PKW}.} {\small I completed this survey while visiting Alma Albujer at the University of C\'{o}rdoba in March, 2011 and would like to warmly thank Alma Albujer, Magdalena Caballero, and Luis Al\'{i}as for their hospitality and support.}
\newpage
\tableofcontents
\setcounter{section}{-1}
\section{Shiffman program}
\subsection{Introduction} \label{PL}
The Poincar\'{e} Lemma addresses the question:
\textit{when is a function equal to the derivative of another function, at least locally?} \cite{Mir95}
\begin{theorem}[Poincar\'{e}'s Lemma] \label{poincare} Let $\Omega \subset \mathbb{R}^{2}$ be
a simply connected domain. When two-variables ${\mathcal{C}}^{1}$ vector field
$\mathcal{F}(x,y)=({\mathcal{F}}_{1}, {\mathcal{F}}_{2}): \Omega \rightarrow {\mathbb{R}}^{2}$
passes the curl test $\frac{\partial {\mathcal{F}}_{1}}{\partial y} = \frac{\partial {\mathcal{F}}_{2}}{\partial x}$,
the vector field $\mathcal{F}$ admits a potential function $\phi : \Omega \rightarrow \mathbb{R}$
of class ${\mathcal{C}}^{2}$ satisfying the equality ${\mathcal{F}} = \nabla \phi$ and it becomes a gradient map.
\end{theorem}
The main aim of this survey based on author's work \cite{Lee09b, Lee10a, Lee10b, Lee09a, Lee11a} is to illustrate geometric applications of Poincar\'{e}'s Lemma to constant mean curvature equations. It implements Shiffman's 1956 proposal (section \ref{SP}).
\subsection{Cauchy--Riemann equations} \label{LCR}
We begin with two-variables harmonic functions. Applying Poincar\'{e}'s Lemma to
$\left(\;f_{x}\;\right)_{x} + \left(\;f_{y}\;\right)_{y} =0$, the Laplace equation,
solves the Cauchy--Riemann equations:
\[
f_{x} = g_{y} \quad \text{and} \quad f_{y}=-g_{x}.
\]
The potential function $g$ (up to adding a constant) is called \textit{the conjugate function} of $f$.
The Cauchy--Riemann equations implies that $g$ is also harmonic:
$\left(\;g_{x}\;\right)_{x} + \left(\;g_{y}\;\right)_{y}
= \left(\;-f_{y}\;\right)_{x} + \left(\;f_{x}\;\right)_{y}
= 0$. In this sense, Laplace's equation is self-dual.
Also, the \textit{combined} complex-valued function $f+ig$ becomes a holomorphic function.
The Cauchy--Riemann equations admit various several-variables generalizations \cite{Pal65, RK50, SW60, SW68} including the Hodge equations.
\subsection{Minimal surfaces and maximal surfaces} \label{MSE}
\subsubsection{One hundred years of Bernstein problems}
Bernstein proved a truly beautiful theorem that the only entire minimal graphs in Euclidean space ${\mathbb{R}}^{3}$
are planes. Moser \cite{Mos61} showed that for all dimensions $n \geq 3$ the entire
solutions of minimal hypersurface equation in ${\mathbb{R}}^{n}$
with bounded gradients have to be affine. Ecker and Huisken \cite{EH90} deduced a curvature estimate theorem to improve Moser's Theorem. Bombieri, De Giorgi and Giusti \cite{BGG69} disproved Bernstein's conjecture for entire minimal hypersurfaces in higher dimensions. Pino, Kowalczyk, and Wei \cite{PKW} used the existence of non-trivial entire minimal graphs in higher dimensions to construct counterexamples of De Giorgi conjecture for the Allen--Cahn equation. For more updates, we refer to \cite{NS05, Pac10, Sim97}.
\subsubsection{Forty years of Calabi problems} A maximal surface in Lorentz space ${\mathbb{L}}^{3}$
with signature $(+,+,-)$ is a spacelike surface with zero mean curvature. Calabi's Theorem that the only entire maximal graphs in ${\mathbb{L}}^{3}$ are spacelike planes and its parametric version that the only complete maximal surfaces in $\mathbb{L}^{3}$ are spacelike planes admit various approaches \cite{AP01b, Cal70, CY76, ER92, ER94, Kob83, MAR10, RR10, Ro96}.
The global theory of maximal surfaces with \textit{singularities} is rich \cite{FL07, FL10, FLS05, FLS07, FRUYY09,
FSUY08, MY05}. Fern\'{a}ndez, L\'{o}pez, and Souam \cite{FLS05} proved that complete embedded maximal surfaces in ${\mathbb{L}}^{3}$ with a finite number of singularities are entire maximal graphs with conelike singularities and
asymptotic to half catenoids or spacelike planes. Conelike singularity corresponds to the point where the Gauss curvature
blows up \cite{Eck86, FLS05, KM03, Kob84}. Cheng and Yau \cite{CY76} extended Calabi's Theorem for all dimensions $n \geq 3$.
\subsubsection{Shiffman proposal and Calabi correspondence} \label{SP} In 1970, Calabi introduced an interesting
duality between minimal graphs in ${\mathbb{R}}^{3}$ and maximal graphs in ${\mathbb{L}}^{3}$. Calabi's
correspondence can be viewed as a geometric generalization of the Cauchy--Riemann equations. The minimal surface equation in ${\mathbb{R}}^{3}$ reads
\[
0 = \frac{\partial}{\partial x} \left( \frac{f_{x}}{\sqrt{1+{f_{x}}^{2}+{f_{y}}^{2}}} \right)
+ \frac{\partial}{\partial y} \left( \frac{f_{y}}{\sqrt{1+{f_{x}}^{2}+{f_{y}}^{2}}} \right).
\]
We apply Poincar\'{e}'s Lemma to the minimal surface equation to obtain the potential function $g$ satisfying the integrability
condition
\[
\left( - \frac{f_{y}}{\sqrt{1+{f_{x}}^{2}+{f_{y}}^{2}}}, \frac{f_{x}}{\sqrt{1+{f_{x}}^{2}+{f_{y}}^{2}}} \right) = \left( {g}_{x}, {g}_{y} \right).
\]
The existence of such potential function $g$ is a classical result, for instance, see Nitsche's book \cite{Nit88}.
In 1956, Shiffman \cite{Shi56} indicated that the function $g$ obeys
\[
0 = \frac{\partial}{\partial x} \left( \frac{g_{x}}{\sqrt{1-{g_{x}}^{2}-{g_{y}}^{2}}}
\right) + \frac{\partial}{\partial y} \left( \frac{g_{y}}{\sqrt{1-{g_{x}}^{2}-{g_{y}}^{2}}}
\right), \quad 1-{g_{x}}^{2}-{g_{y}}^{2} >0.
\]
He proposed that "\textit{In addition, it is possible to study systems of partial differential equations analogous to
the Cauchy--Riemann equations}". In 1966, Jenkins and Serrin \cite{JS66} called the potential function $g$ \textit{the conjugate function} of $f$. In 1970, Calabi \cite{Cal70} stated that the above equation is the \textit{maximal surface equation} in   ${\mathbb{L}}^{3}$. Conversely, applying Poincar\'{e}'s Lemma to the above maximal surface equation in ${\mathbb{L}}^{3}$, we can recover the function $f$ such that
\[
\left( \frac{g_{y}}{\sqrt{1-{g_{x}}^{2}-{g_{y}}^{2}}}, - \frac{g_{x}}{\sqrt{1-{g_{x}}^{2}-{g_{y}}^{2}}} \right) = \left( {f}_{x}, {f}_{y} \right).
\]
This integrability condition implies that the function $f$ satisfies the minimal surface equation in ${\mathbb{R}}^{3}$.
The existence of Calabi's correspondence indicates that the maximal surface equation naturally appears in the study of
the  minimal surface equation, and vice-versa \cite{JS66, Maz07, Maz10, MRT07, MT08}.
It turns out that Shiffman's program is indeed useful. For instance, to construct a properly embedded minimal surface in the
flat product space ${\mathbb{R}}^{2} \times {\mathbb{S}}^{1}$ which is quasi-periodic (but not periodic), Mazet
and Traizet \cite{MT08} solved a Dirichlet boundary value problem on an infinite strip for the maximal surface
equation with prescribed singularities along a line of the strip to solve
to a Jenkins-Serrin type problem for the minimal surface equation.
\begin{example} [Helicoids and catenoids]
Applying Calabi's correspondence to the helicoid $z=\arctan \left( \frac{y}{x} \right)$ as a minimal surface in
$\mathbb{R}^{3}$ yields the Lorentz catenoid $z={\sinh}^{-1} \left( \sqrt{x^{2}+y^{2}} \right)$ as a
maximal surface in $\mathbb{L}^{3}$. On the other hand, Calabi's correspondence transforms the helicoid $z=\arctan \left( \frac{y}{x} \right)$ as a maximal surface in $\mathbb{L}^{3}$ to the catenoid $z={\cosh}^{-1} \left( \sqrt{x^{2}+y^{2}} \right)$ in $\mathbb{R}^{3}$. The Euclidean catenoid $z=\sqrt{{\cosh}^{2}y - x^{2}}$ in $\mathbb{R}^{3}$ corresponds to the helicoid of the second kind in $\mathbb{L}^{3}$ given by $z=x\tanh y$.
(Here, we omitted the description of the domain of height functions of minimal graphs and maximal graphs.)
\end{example}
\begin{example} [Scherk's surfaces]
Under Calabi's correspondence, the doubly periodic Scherk's minimal surface $z=\ln \left( \frac{\cos y}{\cos x} \right)$
transforms to the triply periodic maximal surface $z=\arcsin \left( \sin x \sin y \right)$. The doubly periodic Scherk's minimal surface $\cos z= e^x \cos y$ also corresponds to the singly periodic maximal surface $\sinh z=e^{x} \cos y$. The singly periodic Scherk's minimal surface $ z= \arcsin \left( \sinh x \sinh y \right) $ corresponds to the maximal surface $ z=\ln \left( \frac{\cosh y}{\cosh x} \right)$. (Here, we again omitted the description of the domain of height functions.)
\end{example}
\subsubsection{L\'{o}pez--L\'{o}pez--Souam correspondence}
In 2000, L\'{o}pez, L\'{o}pez, and Souam \cite{LLS00} observed that \textit{a maximal immersion of a Riemann surface
in $\mathbb{L}^{3}$ induced by a triple of holomorphic null one forms $({\phi}_{1}, {\phi}_{2}, -i{\phi}_{3})$
determines a minimal immersion of the Riemann surface in $\mathbb{R}^{3}$ induced by the triple of one
forms $({\phi}_{1},{\phi}_{2}, {\phi}_{3})$.} Ara\'{u}jo and Leite \cite{AL09} revealed the structure of the L\'{o}pez--L\'{o}pez--Souam correspondence. Under this duality, a doubly-periodic Scherk's minimal surface in $\mathbb{R}^{3}$ determines a two-parameter family of non-congruent maximal surfaces in ${\mathbb{L}}^{3}$.
\subsubsection{Two correspondences are the same.} \label{thesame} It is then natural to ask whether or not two dualities between minimal surfaces in ${\mathbb{R}}^{3}$ and maximal surfaces in ${\mathbb{L}}^{3}$ are essentially identical. In \cite{Lee10a}, we proved that the Calabi correspondence becomes the L\'{o}pez--L\'{o}pez--Souam correspondence.
\subsection{Albujer--Al\'{i}as correspondence} \label{ZMC}
In 2008, Albujer and Al\'{i}as \cite{AA08} proved that Calabi's duality admit natural extensions to
zero mean curvature surfaces in product spaces. Let ${\mathcal{M}}_{\kappa}$ endowed with metric ${d\sigma}^{2}$ denote two dimensional space form with curvature $\kappa$. They constructed the duality between graphs with zero mean curvature
in Riemannian product space ${\mathcal{M}}_{\kappa} \times \mathbb{R}$ endowed with the metric
${d\sigma}^{2}+dz^{2}$ and spacelike graphs with zero mean curvature
in Lorentzian product space ${\mathcal{M}}_{\kappa} \times \mathbb{R}_{1}$
equipped with the metric ${d\sigma}^{2}-dz^{2}$.
\newpage
\section{Twin correspondence I} \label{SecTwin1}
\subsection{Introduction} 
We generalize Calabi's correspondence to constant mean curvature graphs in Riemannian and Lorentzian
Bianchi--Cartan--Vranceanu spaces. When the target equation is the minimal surface equation,
which is zero divergence form, we are able to directly
apply Poincar\'{e}'s Lemma to the minimal surface equation to associate its potential function.
The main point of our generalized Calabi correspondence is the fact that the argument using
Poincar\'{e}'s Lemma can be extended to non-zero constant mean curvature equation, which is evidently
non-zero divergence form.
The key idea is to exploit our base curvature lemma (Lemma \ref{bcl}) to transform non-zero
constant mean curvature equation into zero divergence equation. Then applying Poincar\'{e}'s Lemma
to constant mean curvature equation yields:
\begin{theorem}[Twin correspondence I \cite{Lee10a}] \label{DT}
We have the twin correspondence (up to vertical translations) between graphs of constant mean curvature $H$ in
Riemannian Bianchi--Cartan--Vranceanu space $\mathbb{E}^{3}(\kappa, \tau)$ and spacelike graphs of constant mean curvature $\tau$ in Lorentzian Bianchi--Cartan--Vranceanu space $\mathbb{L}^{3}(\kappa, H)$.
\end{theorem}
Our twin correspondence in Theorem \ref{DT} shows us a mysterious reciprocity of two looks-unrelated curvatures: mean
curvature of \textit{surfaces} and bundle curvature of \textit{spaces}. The key idea of the
proof of Theorem \ref{DT} is to exploit a geometric observation on base curvature of \textit{spaces}.
\begin{example}
When both the bundle curvature and the mean curvature vanish, the twin correspondence
reduces to the Albujer--Al\'{i}as duality \cite{AA08} between minimal graphs in Riemannian product space
${\mathcal{M}}_{\kappa} \times \mathbb{R}=\mathbb{E}^{3}(\kappa, 0)$ and maximal graphs in Lorentzian product space ${\mathcal{M}}_{\kappa} \times \mathbb{R}_{1}=\mathbb{L}^{3}(\kappa, 0)$. When the base curvature vanishes,
the twin correspondence induces the duality between graphs of constant mean curvature $H \neq 0$ in Riemannian Heisenberg group ${\mathrm{Nil}}^{3}(\tau)=\mathbb{E}^{3}\left(0, \tau\right)$ and spacelike graphs of constant mean curvature $\tau \neq 0$ in Lorentzian Heisenberg group ${\mathrm{Nil}}^{3}_{1}(H)=\mathbb{L}^{3}\left(0, H\right)$.
\end{example}
\begin{remark}
In recent years there is a high activity on the study of surfaces with constant mean curvature in Bianchi--Cartan--Vranceanu space $\mathbb{E}^{3}(\kappa, \tau)$. In 2009, Daniel, Hauswirth, and Mira wrote a neat lecture note \cite{DHM09} containing recent breakthroughs. We also recommend the survey \cite{FM10} presented by Fern\'{a}ndez and Mira at ICM 2010.
\end{remark}
\subsection{Bianchi--Cartan--Vranceanu spaces}
Our ambient spaces are Riemannian and Lorentzian Bianchi--Cartan--Vranceanu spaces.
The Riemannian Bianchi--Cartan--Vranceanu space $\mathbb{E}^{3}(\kappa, \tau)$ with base curvature $\kappa$ and bundle curvature $\tau$ is
\[
\mathbb{E}^{3}(\kappa, \tau) =\left(\mathbf{V},
\frac{dx^2 +dy^2 }{ \left[ 1+ \frac{\kappa}{4} \left( x^2 + y^2 \right) \right]^{2} } +
\left[ \tau \left( \frac{y dx - x dy}{ 1+ \frac{\kappa}{4} \left( x^2 + y^2 \right) } \right) + dz \right]^2 \right),
\]
where $\mathbf{V}=\{ (x,y,z) \in {\mathbb{R}}^{3} \; \vert \; {\delta}_{\kappa}(x,y)>0 \}$
and ${\delta}_{\kappa}(x,y)= 1+ \frac{\kappa}{4} \left( x^2 + y^2 \right)$.
Up to homothetical normalizations, the Riemannian Bianchi--Cartan--Vranceanu spaces isometric to the following homogeneous manifolds \cite{Dan07, DHM09, FM10}:
\begin{center}
\begin{tabular}{ c | c c c }
{$\mathbf{{\mathbb{E}}^{3}(\kappa, \tau)}$} & $\kappa<0$ & $\kappa=0$ & $\kappa>0$ \\ \hline
$\tau =0$ & {${\mathbb{H}^{2}} \times \mathbb{R}$} & { $ { \mathbb{E}}^{3} $} & {${\mathbb{S}}^{2} \times \mathbb{R} $ } \\
$\tau \neq 0$ & { $ \widetilde{ {PSL}_{2}[ \mathbb{R}]}$} & { ${ \mathrm{Nil} }^{3}$} & { \textbf{${ \hat{\mathbb{S}} }^{3}_{{}_{ { \tiny{\textbf{BERGER}} }}}$} } \\
\end{tabular}
\end{center}
The notion of graphs over ${\mathcal{M}}_{\kappa}=\mathbb{E}^{3}(\kappa, \tau) \cap \{z=0\}$ is natural. Indeed our ambient space $\mathbb{E}^{3}(\kappa, \tau)$ admits a Riemannian fibration over the base space form
\[
{\mathcal{M}}_{\kappa}=\left( \{ (x,y) \in {\mathbb{R}}^{2} \; \vert \; {\delta}_{\kappa}(x,y)>0 \},
\frac{1}{{{\delta}_{\kappa} }^{2}} \left( dx^2 + dy^2 \right) \right)
\]
via the standard projection $(x,y,z) \mapsto (x,y)$. We also introduce the Lorentzian Bianchi--Cartan--Vranceanu space
\[
\mathbb{L}^{3}(\kappa, \tau) =\left( \mathbf{V}, \;
\frac{dx^2 +dy^2 }{ \left[ 1+ \frac{\kappa}{4} \left( x^2 + y^2 \right) \right]^{2} } -
\left[ \tau \left( \frac{y dx - x dy}{ 1+ \frac{\kappa}{4} \left( x^2 + y^2 \right) } \right) + dz \right]^2 \right),
\]
where $\mathbf{V}=\{ (x,y,z) \in {\mathbb{R}}^{3} \; \vert \; {\delta}_{\kappa}(x,y)>0 \}$
and ${\delta}_{\kappa}(x,y)= 1+ \frac{\kappa}{4} \left( x^2 + y^2 \right)$.
\subsection{Base curvature lemma} \label{sectionBCL}
The key idea to prove the twin correspondence in Bianchi--Cartan--Vranceanu spaces
is to exploit the following lemma.
\begin{lemma}[Base curvature lemma \cite{Lee10a}] \label{bcl}
Let $\kappa \in {\mathbb{R}}$ be a constant. The quadratic function $\delta={\delta}_{\kappa}(x,y)=1 + \frac{\kappa}{4} \left( x^2 + y^2 \right)$ satisfies the algebraic identity
\[
\frac{2}{{\delta}^{2}} = \frac{\partial}{\partial x} \left( \frac{x}{\delta} \right) + \frac{\partial}{\partial y} \left( \frac{y}{\delta} \right).
\]
\end{lemma}
\begin{proof}
It is trivial.
\end{proof}
\begin{remark}
In the proof of Theorem \ref{DT}, we shall employ the base curvature lemma
to transform the non-zero constant mean curvature equations in Bianchi--Cartan--Vranceanu spaces to zero divergence type equations.
\end{remark}
\begin{remark}
When $\kappa \neq 0$, the identity in Lemma \ref{bcl} admits a \textit{geometric} motivation.
Let's consider the conformal metric in the $xy$-plane
\[
ds^{2} = \frac{1}{{\delta}^{2}} \left( dx^{2}+dy^{2} \right),
\]
where $\delta={\delta}_{\kappa}(x,y)=1 + \frac{\kappa}{4} \left( x^2 + y^2 \right)$. Its induced
curvature is $\kappa$. Hence the Gauss formula
\[
\kappa = {\delta}^{2} \triangle \left( \ln \delta \right),
\quad \triangle:=\frac{\partial^{2}}{\partial x ^{2}}+\frac{\partial^{2}}{\partial y ^{2}},
\]
can be re-written as in the form
\[
\kappa = {\delta}^{2} \left[ \frac{\partial}{\partial x} \left( \frac{{\delta}_{x}}{\delta} \right)
+ \frac{\partial}{\partial y} \left( \frac{{\delta}_{y}}{\delta} \right) \right]
= \frac{ \kappa {\delta}^{2}}{2} \left[ \frac{\partial}{\partial x} \left( \frac{x}{\delta} \right)
+ \frac{\partial}{\partial y} \left( \frac{y}{\delta} \right) \right].
\]
Canceling out $\kappa\neq0$ in both terms, we meet the desired identity. This argument explains that why Lemma \ref{bcl}
is called the base curvature lemma.
\end{remark}
\subsection{Proof of twin correspondence I}
We first state the constant mean curvature equations in Bianchi--Cartan--Vranceanu spaces.
\begin{prop}
Let the quadruple $\left(\kappa, \tau, H, \epsilon \right) \in \mathbb{R}^{3} \times \{1, -1 \}$ be given. We introduce the partial differential equation $\mathrm{CMC}_{\left(\kappa, \tau, H, \epsilon \right)} [ f(x,y)]$:
\[
\frac{2H}{{\delta}^{2}} = \frac{\partial}{\partial x} \left( \frac{\alpha}{\omega} \right) + \frac{\partial}{\partial y} \left( \frac{\beta}{\omega} \right),
\]
where $\left( \alpha, \beta \right) =\left( f_{x} + \tau \frac{y}{\delta}, f_{y} - \tau \frac{x}{\delta} \right)$ and
$\omega= \sqrt{ 1+ \epsilon {\delta}^{2} \left( {\alpha}^2 + {\beta}^2 \right)}$. Here, for the case when $\epsilon=-1$, we need to assume the \textit{spacelike} condition:
\[
1+\epsilon{\delta}^{2}\left( {\alpha}^2 + {\beta}^2 \right)>0.
\]
\textbf{(a)} The graph $z=f(x,y)$ in $\mathbb{E}^{3}(\kappa, \tau)$ has constant mean curvature $H$ only when
the height function $f$ satisfies the differential equation $\mathrm{CMC}_{\left(\kappa, \tau, H, 1 \right)}$. \\
\textbf{(b)} We say that a surface in $\mathbb{L}^{3}(\kappa, \tau)$ is spacelike if its induced pullback metric is positive definite or equivalently if it admits timelike pointing normal.
The spacelike graph $z=f(x,y)$ in $\mathbb{L}^{3}(\kappa, H)$ has constant mean curvature $\tau$ if and only if
the height function $f$ satisfies the differential equation $\mathrm{CMC}_{\left(\kappa, -H, \tau, -1 \right)}$.
\end{prop}
We restate Theorem \ref{DT} more explicitly:
\begin{theorem}[Twin correspondence I - alternative version \cite{Lee10a}] \label{DT2}
Let $\Omega \subset {\mathcal{M}}_{\kappa}$ be a simply connected domain.\\
\textbf{(a)} We have the twin correspondence (up to adding a constant) between the solution $f: \Omega \rightarrow \mathbb{R}$ of the constant mean curvature equation $\mathrm{CMC}_{\left(\kappa, \tau, H, 1 \right)}$
\[
\frac{2H}{{\delta}^{2}} = \frac{\partial}{\partial x} \left( \frac{ \alpha }{ \sqrt{ 1+ {\delta}^{2} \left( {\alpha}^2 + {\beta}^2 \right)}} \right) + \frac{\partial}{\partial y} \left( \frac{\beta}{ \sqrt{ 1+ {\delta}^{2} \left( {\alpha}^2 + {\beta}^2 \right)}} \right)
\]
where $\left( \alpha, \beta \right) =\left( f_{x} + \tau \frac{y}{\delta}, f_{y} - \tau \frac{x}{\delta} \right)$ and the solution $g: \Omega \rightarrow \mathbb{R}$ of the constant mean curvature equation $\mathrm{CMC}_{\left(\kappa, -H, \tau, -1 \right)}$
\[
\frac{2 \tau}{{\delta}^{2}} = \frac{\partial}{\partial x} \left( \frac{ \tilde{\alpha}}{ \sqrt{ 1-{\delta}^{2} \left( {\tilde{\alpha}}^2 + {\tilde{\beta}}^2 \right)}} \right) + \frac{\partial}{\partial y} \left( \frac{\tilde{\beta}}{ \sqrt{ 1- {\delta}^{2} \left( {\tilde{\alpha}}^2 + {\tilde{\beta}}^2 \right)}} \right),
\]
where $\left( \tilde{\alpha}, \tilde{\beta} \right) =\left( g_{x} - H \frac{y}{\delta}, g_{y} + H \frac{x}{\delta} \right)$.
Here we assume that $g$ satisfies the spacelike condition $1- {\delta}^{2} \left( {\tilde{\alpha}}^2 + {\tilde{\beta}}^2 \right)>0$. \\
\textbf{(b)} The twin correspondence is obtained by solving the twin relation:
\[
\left( \tilde{\alpha} , \tilde{\beta} \right) = \left( - \frac{ \beta }{ \omega },
\frac{ \alpha }{ \omega} \right)
\quad\text{or equivalently}\quad
\left( \alpha , \beta \right) = \left( \frac{ \tilde{\beta} }{ \tilde{\omega} },
- \frac{ \tilde{\alpha} }{\tilde{\omega} } \right),
\]
where $\omega= \sqrt{ 1+ {\delta}^{2} \left( {\alpha}^2 + {\beta}^2 \right)}$ and $\tilde{\omega}=\sqrt{ 1- {\delta}^{2} \left( {\tilde{\alpha}}^2 + {\tilde{\beta}}^2 \right)}$. \\
\textbf{(c)} Since two integrability conditions in (b) are equivalent, the twin correspondence is involutive. The twin minimal surface $\widehat{\widehat{\Sigma}}$ of the twin surface $\widehat{\Sigma}$ of a minimal surface ${\Sigma}$
is congruent to ${\Sigma}$ up to vertical translations. We say that the graph $z=f(x,y)$ is the twin surface of the graph $z=g(x,y)$, and vice-versa.
\end{theorem}
\begin{proof}
Let $f: \Omega \rightarrow \mathbb{R}$ be a solution of $\mathrm{CMC}_{\left(\kappa, \tau, H, 1 \right)}$. The main idea is to employ the base curvature lemma (Lemma \ref{bcl} in Section \ref{sectionBCL})
\[
\frac{2}{{\delta}^{2}} = \frac{\partial}{\partial x} \left( \frac{x}{\delta} \right) + \frac{\partial}{\partial y} \left( \frac{y}{\delta} \right)
\]
to rewrite $\mathrm{CMC}_{\left(\kappa, \tau, H, 1 \right)}$:
\[
\frac{2H}{{\delta}^{2}} = \frac{\partial}{\partial x} \left( \frac{\alpha}{\omega} \right) + \frac{\partial}{\partial y} \left( \frac{\beta}{\omega} \right)
\]
in the divergence-zero form
\[
0 = \frac{\partial}{\partial x} \left( \frac{\alpha}{\omega} - H \frac{x}{\delta} \right)
+ \frac{\partial}{\partial y} \left( \frac{\beta}{\omega} - H \frac{y}{\delta} \right).
\]
We now use Poincar\'{e}'s Lemma to associate a function $g: \Omega \rightarrow \mathbb{R}$ such that
\[
\left( g_{x}, g_{y} \right) = \left( -\frac{\beta}{\omega} + H \frac{y}{\delta}, \frac{\alpha}{\omega} - H \frac{x}{\delta} \right).
\]
It thus follows that
\[
\left( \tilde{\alpha} , \tilde{\beta} \right) = \left( \frac{-\beta }{ \omega },
\frac{ \alpha }{ \omega} \right)
\]
or
\[
\left( \alpha , \beta \right) = \left( \frac{ \tilde{\beta} }{ \tilde{\omega} },
\frac{ - \tilde{\alpha} }{\tilde{\omega} } \right).
\]
We now combine the definition
\[
\left( \alpha, \beta \right) =\left( f_{x} + \tau \frac{y}{\delta}, f_{y} - \tau \frac{x}{\delta} \right),
\]
the integrability condition
\[
\frac{\partial}{\partial x} \left( f_{y} \right) - \frac{\partial}{\partial y} \left( f_{x} \right)=0,
\]
and the base curvature lemma
\[
\frac{2}{{\delta}^{2}} = \frac{\partial}{\partial x} \left( \frac{x}{\delta} \right) + \frac{\partial}{\partial y} \left( \frac{y}{\delta} \right)
\]
to deduce the equality
\[
- {\beta}_{x} + {\alpha}_{y} = - \frac{\partial}{\partial x} \left( f_{y} - \tau \frac{x}{\delta} \right) + \frac{\partial}{\partial y} \left(
f_{x} + \tau \frac{y}{\delta} \right) =\frac{2 \tau}{{\delta}^{2}}.
\]
Then by the twin relation
\[
\left( \alpha , \beta \right) = \left( \frac{ \tilde{\beta} }{ \tilde{\omega} },
\frac{ - \tilde{\alpha} }{\tilde{\omega} } \right),
\]
this becomes the equality
\[
\frac{\partial}{\partial x} \left( \frac{ \tilde{\alpha} }{\tilde{\omega} } \right) + \frac{\partial}{\partial y}
\left( \frac{ \tilde{\beta} }{ \tilde{\omega} } \right) = \frac{2 \tau}{{\delta}^{2}}.
\]
This means that the function $g: \Omega \rightarrow \mathbb{R}$ satisfies the $\mathrm{CMC}_{\left(\kappa, -H, \tau, -1 \right)}$. Working backwards, we also obtain the other direction, i.e. from $\mathrm{CMC}_{\left(\kappa, -H, \tau, -1 \right)}$ to $\mathrm{CMC}_{\left(\kappa, \tau, H, 1 \right)}$.
\end{proof}
\begin{corollary}[\cite{Lee10a}] \label{Bernstein}
We have the twin correspondence between the moduli space of \textit{entire} graphs with constant mean curvature $H$ in
$\mathbb{E}^{3}(\kappa, \tau)$ and that of \textit{entire} spacelike graphs with constant mean curvature $\tau$ in
$\mathbb{L}^{3}(\kappa, H)$.
\end{corollary}
\begin{proof}
Take $\Omega={\mathcal{M}}_{\kappa}$ in Theorem \ref{DT} or Theorem \ref{DT2}.
\end{proof}

\subsection{Twin surfaces in simultaneous conformal coordinates} \label{reaingWeir}
We can show the existence of simultaneous conformal coordinates for
the graphs of constant mean curvature $H$ in
$\mathbb{E}^{3}(\kappa, \tau)$ and
its twin graphs of constant mean curvature $\tau$ in $\mathbb{L}^{3}(\kappa, H)$.
\begin{prop}[\cite{Lee09a}] \label{cdn}
Let $\zeta=\xi + i \eta $ denote the conformal parameter of the graph $z=f(x,y)$ of constant mean curvature $H$ in $\mathbb{E}^{3}(\kappa, \tau)$.
Then, $\zeta=\xi + i \eta$ also becomes the conformal parameter of its twin graph $z=g(x,y)$ of constant mean curvature $\tau$ in $\mathbb{L}^{3}(\kappa, H)$.
\end{prop}
\begin{proof} We write $F(x,y)=(x,y,f(x,y)$ and $G(x,y)=(x,y,g(x,y))$. We consider a coordinate transformation
$\phi:(x,y) \rightarrow \left( \xi, \eta \right)$
to obtain the conformal parametrization of the graph $z=f(x,y)$ in $\mathbb{E}^{3}(\kappa, \tau)$:
\[
\widehat{F}=F \circ {\varphi}^{-1} : \left( \xi, \eta \right) \mapsto (x,y) \mapsto (x,y, f(x,y))
\]
Our job is to show that
\[
\widehat{G}=G \circ {\varphi}^{-1} : \left( \xi, \eta \right) \mapsto (x,y) \mapsto (x,y, g(x,y))
\]
is also a conformal parametrization of the twin graph $z=g(x,y)$ in $\mathbb{L}^{3}(\kappa, H)$. Let
\[
\left( \xi, \eta \right) \mapsto (x,y)=\left(p \left( \xi, \eta \right) , q \left( \xi, \eta \right) \right)
\]
denote the inverse coordinate transformation ${\phi}^{-1}: \left( \xi, \eta \right) \rightarrow (x,y)$.
We prepare the Riemannian orthonormal frame in $\mathbb{E}^{3}(\kappa, \tau)$
\[
E_{1}= {\delta}_{\kappa}{\partial}_{x} - \tau y {\partial}_{z}, \;\; E_{2}= {\delta}_{\kappa}{\partial}_{y} + \tau x {\partial}_{z}, \;\; E_{3}= {\partial}_{z}
\]
and the Lorentzian orthonormal frame in $\mathbb{L}^{3}(\kappa, H)$
\[
L_{1}= {\delta}_{\kappa}{\partial}_{x} - H y {\partial}_{z},\;\; L_{2}={\delta}_{\kappa}{\partial}_{y} + H x {\partial}_{z}, \;\; L_{3}= {\partial}_{z}.
\]
With the complexifed differential operator $\frac{\partial}{\partial \zeta}= \frac{1}{2} \left( \frac{\partial}{\partial \xi} -i \frac{\partial}{\partial \eta} \right)$, the Chain Rule gives
\[
\frac{\partial \widehat{F}}{\partial \zeta} = \frac{1}{ {\delta}_{\kappa} } \frac{\partial p}{\partial \zeta} E_{1} + \frac{1}{ {\delta}_{\kappa} } \frac{\partial q}{\partial \zeta} E_{2} + \left( \alpha \frac{\partial p}{\partial \zeta} + \beta \frac{\partial q}{\partial \zeta} \right) E_{3}
\]
and
\[
\frac{\partial \widehat{G}}{\partial \zeta} = \frac{1}{ {\delta}_{\kappa} } \frac{\partial p}{\partial \zeta} L_{1} + \frac{1}{ {\delta}_{\kappa} } \frac{\partial q}{\partial \zeta} L_{2} + \left(\tilde{\alpha} \frac{\partial p}{\partial \zeta} + \tilde{\beta} \frac{\partial q}{\partial \zeta} \right) L_{3}.
\]
Since $\widehat{F}$ is the isothermal parametrization, we see that $ \langle \frac{\partial \widehat{F}}{\partial \zeta}, \frac{\partial \widehat{F}}{\partial \zeta} \rangle$ vanishes. To prove that $\widehat{G}$ becomes the isothermal parametrization, we verify that $\langle \frac{\partial \widehat{G}}{\partial \zeta}, \frac{\partial \widehat{G}}{\partial \zeta} \rangle$ also vanishes. Using the twin relation
$ \left( \tilde{\alpha} , \tilde{\beta} \right) = \left( - \frac{ \beta }{ \omega }, \frac{ \alpha }{ \omega} \right)$, we deduce
\[
\langle \frac{\partial \widehat{G}}{\partial \zeta}, \frac{\partial \widehat{G}}{\partial \zeta} \rangle
=  \frac{1}{{\omega}^{2}} \langle \frac{\partial \widehat{F}}{\partial \zeta}, \frac{\partial \widehat{F}}{\partial \zeta} \rangle.
\]
Indeed, we compute
\[\label{lasces}
\begin{array}{lll}
&& \langle \frac{\partial \widehat{G}}{\partial \zeta}, \frac{\partial \widehat{G}}{\partial \zeta} \rangle \\
&=& \left( \frac{1}{{\delta}_{\kappa}} \frac{\partial p}{\partial \zeta} \right)^{2} + \left( \frac{1}{{\delta}_{\kappa}} \frac{\partial q}{\partial \zeta} \right)^{2}
- \left( \tilde{\alpha} \frac{\partial p}{\partial \zeta} + \tilde{\beta} \frac{\partial q}{\partial \zeta} \right)^{2} \\
&=& \left( \frac{1- {\tilde{\alpha} }^{2} {{\delta}_{\kappa}}^{2} }{{{\delta}_{\kappa}}^{2} } \right) { \left( \frac{\partial p}{\partial \zeta} \right)}^{2} +
\left( \frac{1-{ \tilde{\beta} }^{2} {{\delta}_{\kappa}}^{2}}{ {{\delta}_{\kappa}}^{2} } \right) {\left( \frac{\partial q}{\partial \zeta} \right) }^{2} - 2 \tilde{\alpha} \tilde{\beta} \frac{\partial p}{\partial \zeta} \frac{\partial q}{\partial \zeta} \\
&=& \left( \frac{1 + {\alpha}^2 {{\delta}_{\kappa}}^{2} }{{\omega}^{2} {{\delta}_{\kappa}}^{2} } \right) {\left( \frac{\partial p}{\partial \zeta} \right) }^{2} + \left( \frac{1 + {\beta}^2 {{\delta}_{\kappa}}^{2} }{{\omega}^{2} {{\delta}_{\kappa}}^{2} } \right) {\left( \frac{\partial q}{\partial \zeta} \right) }^{2} - 2 \left( - \frac{ \alpha \beta}{{\omega}^{2} } \right) \frac{\partial p}{\partial \zeta} \frac{\partial q}{\partial \zeta} \\
&=& \frac{1}{{\omega}^{2}} \left[ \left( \frac{1}{{\delta}_{\kappa}} \frac{\partial p}{\partial \zeta} \right)^{2} + \left( \frac{1}{{\delta}_{\kappa}} \frac{\partial q}{\partial \zeta} \right)^{2}
+ \left( {\alpha} \frac{\partial p}{\partial \zeta} + {\beta} \frac{\partial q}{\partial \zeta} \right)^{2} \right] \\
&=& \frac{1}{{\omega}^{2}} \langle \frac{\partial \widehat{F}}{\partial \zeta}, \frac{\partial \widehat{F}}{\partial \zeta} \rangle.
\end{array}
\]
This completes the proof of Lemma \ref{cdn}.
\end{proof}

 We can read local geometric information of twin surfaces using the simultaneous conformal coordinates.
 
\begin{prop}[\cite{Lee09a}] \label{cdnmore} When the twin correspondence transforms the graph $\Sigma$ of constant mean curvature $H$ in $\mathbb{E}^{3}(\kappa, \tau)$ to the spacelike graph ${\Sigma}^{*}$ of constant mean curvature $\tau$ in $\mathbb{L}^{3}(\kappa, H)$, the metric ${\mathrm{I}}_{{}_{\Sigma}}$ corresponds to the metric  ${\mathrm{I}}_{{}_{{\Sigma}^{*}}}=u^2 {\mathrm{I}}_{{}_{\Sigma}}$, where $u$ denotes the angle function of $\Sigma$ induced by the vertical Killing vector field ${\partial}_{z}$ in $\mathbb{E}^{3}(\kappa, \tau)$. By the twin correspondence, the complexified Gauss map ${g}_{{}_{{\Sigma}}}$   corresponds to  the complexified Gauss map ${g}_{{}_{{\Sigma}^{*}}} = i \, {g}_{{}_{{\Sigma}}}$. 
\end{prop}

\subsection{Sister surfaces and twin surfaces}
Daniel \cite{Dan07, FM10} proved a fascinating result that Lawson's isometric correspondence
for constant mean curvature surfaces in three dimensional space forms extends to Riemannian
Bianchi--Cartan--Vranceanu spaces $\mathbb{E}^{3}(\kappa, \tau)$.
For instance, we have an isometric duality between minimal
surfaces in Riemannian Heisenberg space ${\mathrm{Nil}}^{3}={\mathrm{Nil}}^{3}\left(\frac{1}{2}\right)={\mathbb{E}}^{3}\left(0, \frac{1}{2} \right)$ and surfaces of constant mean curvature $\frac{1}{2}$ in Riemannian product space $\mathbb{H}^{2} \times \mathbb{R}={\mathbb{E}}^{3}\left(-1, 0 \right)$. Two such surfaces related by Daniel's correspondence are called
\textit{sister surfaces}. The following commutative diagram illustrates the twin correspondence and Daniel's sister correspondence simultaneously:
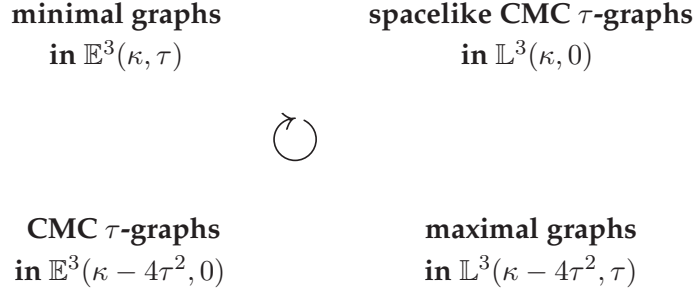
\begin{center}
\begin{tabular}{ c c c }
{ { \textbf{ {minimal graphs} } } } & & { { \textbf{ {spacelike CMC $\tau$-graphs}} } } \\
{ { \textbf{ {in ${\mathbb{E}}^{3}(\kappa, \tau)$} } } }
& & { { \textbf{ {in $ \mathbb{L}^{3}(\kappa, 0)$} } } } \\
& {{\textbf{ }}} & \\
& {\Huge $\circlearrowright$} & \\
& {{\textbf{ }}} & \\
{ { \textbf{ {CMC $\tau$-graphs} }}} & & {{\textbf{ {maximal graphs}}}} \\
{ { \textbf{ {in $\mathbb{E}^{3}( \kappa-4 {\tau}^{2}, 0)$ } }}} & & {{\textbf{ {in ${\mathbb{L}}^{3}(\kappa-4 {\tau}^{2}, \tau)$}}}} \\
\end{tabular}
\end{center}
\begin{remark}
The above commutative diagram induces isometric correspondences for spacelike graphs of constant mean curvature
in Lorentzian Bianchi--Cartan--Vranceanu space.
\end{remark}
\begin{remark}[\textit{complete vs.\;\textit{entire}}]
Daniel and Hauswirth \cite{DH09} showed that a minimal surface in ${\mathrm{Nil}}^{3}$ is an \textit{entire} graph if and only if its sister surface with constant mean curvature $\frac{1}{2}$ in $\mathbb{H}^{2} \times \mathbb{R}$
is an \textit{entire} graph. They showed that a \textit{complete} minimal surface in ${\mathrm{Nil}}^{3}$ which is transverse to
the vertical Killing vector field is an \textit{entire} graph. Under the twin correspondence, an \textit{entire} graph $\Sigma$ with constant mean curvature $H$ in $\mathbb{E}^{3}(\kappa, \tau)$ corresponds to an \textit{entire} spacelike twin graph $\widehat{\Sigma}$ with constant mean curvature $\tau$ in $\mathbb{L}^{3}(\kappa, H)$. By Proposition \ref{cdnmore}, the twin correspondence shrinks the induced metric ${ds_{\widehat{\Sigma}}}^{2}\leq{ds_{\Sigma}}^{2}$. In general, the completeness of entire graphs in Lorentzian Bianchi--Cartan--Vranceanu space is not guaranteed. Indeed, Albujer \cite{Alb09} constructed  entire \textit{incomplete} graphs with zero mean curvature in Lorentzian product space $\mathbb{H}^{2} \times \mathbb{R}_{1}=\mathbb{L}^{3}(-1, 0)$.
\end{remark}
\subsection{Calabi--Chern type problem in Lorentzian Heisenberg space} \label{entireLHG}
The Riemannian Heisenberg space ${\mathrm{Nil}}^{3}(\tau)=\mathbb{E}^{3}\left(0, \tau\right)$ is $\mathbb{R}^{3}$ endowed with the metric
\[
dx^2 +dy^2 + \left[ \tau \left( y dx - x dy \right) + dz \right]^2.
\]
Fern\'{a}ndez and Mira \cite{FM09} parameterized the moduli space of
entire minimal graphs in ${\mathrm{Nil}}^{3}=\mathbb{E}^{3}\left(0, \frac{1}{2}\right)$ using holomorphic quadratic differentials.
There exist lots of entire minimal graphs in Riemannian Heisenberg space ${\mathrm{Nil}}^{3}$ \cite{DH09, FM09, FM10}. We here employ the twin correspondence to show that there exist no entire spacelike graphs with zero mean curvature in Lorentzian Heisenberg space ${\mathrm{Nil}}^{3}_{1}(\tau)=\mathbb{L}^{3}\left(0, \tau\right)$, which is $\mathbb{R}^{3}$ equipped with the metric
\[
dx^2 +dy^2 - \left[ \tau \left( y dx - x dy \right) + dz \right]^2.
\]
\begin{remark}
The affine function $q(x,y)={\mu}_{1}x+{\mu_{2}}y$ satisfies the zero mean curvature equation in ${\mathrm{Nil}}^{3}_{1}\left( \tau \right)$. When $\tau \neq 0$, the spacelike condition rules out entire solutions:
$0 < 1 - \left(q_{x}+\tau y \right)^2 - \left( q_{y} - \tau x \right)^2 = 1 - \left({\mu}_{1}+\tau y \right)^2 - \left( {\mu}_{2}- \tau x \right)^2$.
\end{remark}
\begin{theorem}[Calabi--Chern type problem in Lorentzian Heisenberg space \cite{Lee09a}] \label{ChernLHG}
There is no entire spacelike graph with zero mean curvature in ${\mathrm{Nil}}^{3}_{1}(\tau)$, $\tau \neq 0$.
\end{theorem}
\begin{proof}
We assume to the contrary that there exists an entire spacelike graph
with zero constant mean curvature in ${\mathrm{Nil}}^{3}_{1}(\tau)=\mathbb{L}^{3}(0, \tau)$. By the twin correspondence in
Corollary \ref{Bernstein}, its twin graph with constant mean curvature $\tau$ in
$\mathbb{R}^{3}=\mathbb{E}^{3}(0, 0)$ is entire. By Chern's Theorem \cite{Ch65} that the only entire graphs with constant mean curvature in $\mathbb{R}^{3}$ are minimal (and so they are planes), the mean curvature $\tau$ of twin graph
vanishes, which contradicts for $\tau \neq 0$.
\end{proof}
\begin{remark}
In \cite{Lee10b}, we investigated maximal surfaces in Lorentzian Heisenberg space. We showed that their Gauss maps become
harmonic maps into the standard sphere $\mathbb{S}^{2}$. We also established Kenmotsu type representation formula for maximal surfaces. The holomorphicity of Abresch--Rosenberg type differential on maximal surfaces is proved.
We used the twin correspondence to construct explicit examples of maximal graphs.
\end{remark}
\subsection{Minimal graphs in Riemannian Heisenberg space} \label{graphHG}
We apply the twin correspondence to spacelike surfaces of constant mean curvature $\tau$ in $\mathbb{L}^{3}$
to construct  minimal surfaces in Riemannian Heisenberg space ${\mathrm{Nil}}^{3}(\tau)$.
\begin{example}[{Helicoidal surfaces in $\mathbb{L}^{3}$ and catenoids in ${\mathrm{Nil}}^{3}\left( \tau \right)$}]
Let $\lambda \geq 0$. We begin with the helicoidal surface with constant mean curvature $\tau$ in ${\mathbb{L}}^{3}$
\[
z=g(x,y)= \lambda \arctan{\left( \frac{y}{x} \right)} + \tau h \left( \sqrt{x^2 +y^2} \right).
\]
Its twin surface $z=f(x,y)=\lambda \rho\left( \sqrt{x^2 +y^2} \right)$ is the half catenoid in ${\mathrm{Nil}}^{3}\left( \tau \right)$ over $\sqrt{x^2 +y^2} \geq \lambda$. Two functions $h: [\lambda, \infty) \rightarrow \mathbb{R}$ and $\rho: [\lambda, \infty) \rightarrow \mathbb{R}$
satisfy
\[
\frac{d}{dt} h \left( t \right) = \sqrt{ \frac{ t^{2} - {\lambda}^{2} }{{\tau}^{2} t^2 + 1 } }, \quad
\frac{d}{dt} \rho \left( t \right) = \sqrt{ \frac{{\tau}^{2} t^2 + 1 }{ t^{2} - {\lambda}^{2} } }.
\]
The catenoid in ${\mathrm{Nil}}^{3}\left( \tau \right)$ admits the following conformal parametrization.
\[
X(u, \theta)= \left(\lambda \cosh \left( w(u) \right) \cos \left( \theta + \tau \lambda u \right),
\lambda \cosh \left( w(u) \right) \sin \left( \theta + \tau \lambda u \right),
\lambda \rho( w(u) ) \right)
\]
The one-variable function $w(u)$ satisfies the ordinary differential equation
\[
\frac{d}{du} w \left( u \right) = \sqrt{{\tau}^{2} {\lambda}^{2} \cosh^{2} \left( w(u) \right) +1 }.
\]
\end{example}
\begin{example}[{Elliptic Delaunay surfaces in $\mathbb{L}^{3}$ and helicoids in ${\mathrm{Nil}}^{3}\left( \tau \right)$}]
We consider the Delaunay graph $z=g(x,y)=\rho\left( \sqrt{x^2 +y^2} \right)$ with constant mean curvature $\tau$ in ${\mathbb{L}}^{3}$. The radius function $\rho: [\lambda, \infty) \rightarrow \mathbb{R}$ satisfies
\[
2 \tau = \frac{1}{t} \frac{{\rho}'(t) }{\sqrt{ 1 - {{\rho}'(t)}^{2} } } + \frac{d}{dt} \left( \frac{{\rho}'(t) }{\sqrt{ 1 - {{\rho}'(t)}^{2} } } \right).
\]
It is immediate that its first integral is given by, for some $\mu \in \mathbb{R}$,
\[
\frac{{\rho}'(t) }{\sqrt{ 1 - {{\rho}'(t)}^{2} } } = \tau t - \frac{\mu}{t}.
\]
Its twin minimal graph is the helicoid $z=f(x,y)= \mu \arctan{\left( \frac{y}{x} \right)}$ in ${\mathrm{Nil}}^{3}\left( \tau \right)$.
\end{example}
\begin{example}[{Hyperbolic Delaunay surfaces in $\mathbb{L}^{3}$ and its twin graphs in ${\mathrm{Nil}}^{3}(\tau)$}] \label{daen}
We begin with the Delaunay graph $z=g(x,y)= \sqrt{x^{2} +{\rho(y)}^{2} }$ with constant mean curvature $\tau$ in ${\mathbb{L}}^{3}$. The radius function $\rho: [\lambda, \infty) \rightarrow \mathbb{R}$ satisfies
\[
2 \tau = \frac{ 1 }{{\rho}(t) \sqrt{ 1 - {{\rho}'(t)}^{2} } } + \frac{d}{dt} \left( \frac{{\rho}'(t) }{\sqrt{ 1 - {{\rho}'(t)}^{2} } } \right)
\]
or
\[
2 \tau {\rho}(t) {\rho}'(t) = \frac{d}{dt} \left( \frac{{\rho}(t) }{\sqrt{ 1 - {{\rho}'(t)}^{2} } } \right).
\]
Its first integral is given by the differential equation
\[
\tau {\rho (t)}^{2} - \frac{{\rho}(t) }{\sqrt{ 1 - {{\rho}'(t)}^{2} } } = \lambda
\]
for some $\lambda \in \mathbb{R}$. The twin minimal graph in ${\mathrm{Nil}}^{3}\left( \tau \right)$ is
\[
z=f(x,y)=\left( \frac{{\rho}'(y) }{\sqrt{ 1 - {{\rho}'(y)}^{2} } } -\tau y \right) x.
\]
If $\lambda = - \frac{4}{\tau}$, the corresponding Delaunay surface in ${\mathbb{L}}^{3}$ is called the \textit{Semitrough} \cite{CT90, HTTW95}, which is a complete constant mean curvature $\tau$ surface that invariant under hyperbolic rotations in $\mathbb{L}^{3}$. In this case, its twin graph admits an explicit conformal parametrization and also constructed by Daniel \cite{Dan06}. However, he used the Weierstrass formula of minimal surfaces in Riemannian Heisenberg space with the harmonic Gauss map of the Semitrough.
\end{example}
\begin{remark}
Lifting up the Gauss map of the Semitrough yields a minimal graph $z = - arcsinh \left( \frac{1-x^{2}-y^{2}}{2x} \right)$ in ${\mathbb{H}}^{2} \times \mathbb{R}$. Its level curves are foliated by the equidistant curves in the horizontal slices.
\end{remark}
\begin{example} [{Hyperbolic cylinders in $\mathbb{L}^{3}$ and translation invariant graphs in ${\mathrm{Nil}}^{3}={\mathrm{Nil}}^{3}\left(\frac{1}{2}\right)=\mathbb{E}^{3}\left(0, \frac{1}{2}\right)$}]
We begin with the hyperbolic cylinder $z^2 -y^2 =1$ in $\mathbb{L}^{3}$. It is invariant under the hyperbolic rotation about the
$x$-axis in $\mathbb{L}^{3}$. Applying the hyperbolic rotation about the $y$-axis in $\mathbb{L}^{3}$
\[
\left(x,y,z \right) \mapsto \left( \left(\cosh \theta \right) x + \left(\sinh \theta \right) z, \; y, \; \left(\sinh \theta \right) x
+ \left(\cosh \theta \right) z \right), \; \theta \in \mathbb{R}
\]
to our hyperbolic cylinder, we obtain the spacelike Delaunay surface in $\mathbb{L}^{3}$:
\[
\left[ (\cosh \theta) z - (\sinh \theta) x \right]^2 -y^2 =1.
\]
Let's take the upper part of this Delaunay surface:
\[
z=g^{\theta}(x,y) := \frac{\sqrt{ y^2 + 1} }{\cosh \theta} + \frac{ \sinh \theta }{ \cosh \theta } x.
\]
Its corresponding twin surface in ${\mathrm{Nil}}^{3}$ is given by the entire minimal graph
\[
z=f^{\theta}(x,y) := \frac{1}{2}xy - \frac{\sinh \theta}{2} \left[ {arcsinh} \; y + y\sqrt{1+y^2} \right],
\]
It is clear that it is invariant under the translations along the $x$-axis in ${\mathrm{Nil}}^{3}$:
\[
\left(x,y,z \right) \mapsto \left(x+a, y, z+\frac{1}{2} a y \right), \; a \in \mathbb{R}.
\]
The translation invariant graphs $z=f^{\theta}(x,y)$, $\theta \in \mathbb{R}$ are called saddle type surfaces
\cite{AR05, FMP99}. The graph $z=f^{\theta}(x,y)$ admits the conformal chart
\[
X^{\theta}(u,v) = \left( \left( \cosh \theta \right) u + \left( \sinh \theta \right) \cosh v, \, \sinh v, \, \frac{ \cosh \theta}{2} u \sinh v
- \frac{\sinh \theta }{2} v \right)
\]
with the conformal metric $ds^{2}= \cosh^{2} \theta \cosh^{2} v \left( du^2 +dv^2 \right)$.
\end{example}
\newpage
\section{Twin correspondence II} \label{SecTwin2}
\subsection{Introduction}
We introduce the notion of area angle.
\begin{definition} Given a function $\phi=\left({\phi}_{1}, \cdots, {\phi}_{n}\right): \Omega \rightarrow {\mathbb{R}}^{n}$, $n \geq 2$, we write
\[
\Vert \mathcal{J} \Vert := \sqrt{ \sum_{1 \leq i<j \leq n} {{\mathcal{J}}_{i,j}}^{2}}, \quad {\mathcal{J}}_{i,j}:= \frac{ \partial \left( {\phi}_{i}, {\phi}_{j} \right)}{\partial (x, y)}.
\]
When the estimation $1 > \Vert \mathcal{J} \Vert$ holds on $\Omega$, we say that $\phi=\left({\phi}_{1}, \cdots, {\phi}_{n}\right)$ is area-decreasing. When $\phi$ is  area-decreasing, we call   ${\Theta}_{\phi}=\cos^{-1}{\left( \Vert \mathcal{J} \Vert \right)} \in \left(0, \frac{\pi}{2} \right]$ the area angle function of $\phi$ and say that the graph of $\phi$ admits positive area angle function.
\end{definition}
We extend Calabi's correspondence to higher codimension as follows.
\begin{theorem}[Twin correspondence II \cite{Lee11a}] \label{twintwo}
We have the twin correspondence (up to vertical translations) between two dimensional minimal graphs having (not necessarily constant) \textit{positive area angle functions} in Euclidean space ${\mathbb{R}}^{n+2}$ endowed with the metric ${dx_{1}}^{2} +\cdots + {dx_{n+2}}^{2}$ and two dimensional maximal graphs having \textit{the same positive area angle functions} in pseudo-Euclidean space ${\mathbb{R}}^{n+2}_{n}$ equipped with the metric ${dx_{1}}^{2} + {dx_{2}}^{2} - {dx_{3}}^{2} -\cdots - {dx_{n+2}}^{2}$.
\end{theorem}
\begin{theorem}[Twin correspondence II - alternative version \cite{Lee11a}] \label{twintwo2}
\textbf{(a)} If the graph $\mathbf{\Phi}$ of an area-decreasing map $f=\left({f}_{1}, \cdots, {f}_{n}\right): \Omega \rightarrow {\mathbb{R}}^{n}$ with the area angle function ${\Theta} \in \left(0, \frac{\pi}{2} \right]$ becomes
a minimal surface in ${\mathbb{R}}^{n+2}$, then
there exists an area-decreasing map $g=\left({g}_{1}, \cdots, {g}_{n}\right): \Omega \rightarrow {\mathbb{R}}^{n}$ with
the same area angle function ${\Theta} \in \left(0, \frac{\pi}{2} \right]$ such that its graph $\mathbf{\widehat{\Phi}}$ becomes
a maximal surface in ${\mathbb{R}}^{n+2}_{n}$. \\
\textbf{(b)} Conversely, if the graph $\mathbf{\widehat{\Phi}}$ of an area-decreasing map $g=\left({g}_{1}, \cdots, {g}_{n}\right): \Omega \rightarrow {\mathbb{R}}^{n}$ with the area angle function ${\Theta} \in \left(0, \frac{\pi}{2} \right]$ is a maximal surface in ${\mathbb{R}}^{n+2}_{n}$, then we can associate an area-decreasing map $f=\left({f}_{1}, \cdots, {f}_{n}\right): \Omega \rightarrow {\mathbb{R}}^{n}$ with the same area angle function ${\Theta} \in \left(0, \frac{\pi}{2} \right]$ such that its graph $\mathbf{\Phi}$ is minimal surface in ${\mathbb{R}}^{n+2}$.\\
\textbf{(c)} The twin correspondences $\mathbf{\Phi} \Leftrightarrow \mathbf{\widehat{\Phi}}$ in (a) and (b) fulfill the following properties:\\
\textbf{(c1)} The twin relations hold:
\[
\left( \frac{\partial g_{k}}{\partial x}, \frac{\partial g_{k}}{\partial y} \right)
= \left( - \frac{E}{\omega} \frac{\partial f_{k}}{\partial y} + \frac{F}{\omega} \frac{\partial f_{k}}{\partial x},
\frac{G}{\omega}\frac{\partial f_{k}}{\partial x} - \frac{F}{\omega} \frac{\partial g_{k}}{\partial y} \right), \quad k \in \{1, \cdots, n\}
\]
and equivalently
\[
\left( \frac{\partial f_{k}}{\partial x}, \frac{\partial f_{k}}{\partial y} \right)
= \left( \frac{\widehat{E}}{\widehat{\omega}} \frac{\partial g_{k}}{\partial y} - \frac{\widehat{F}}{\widehat{\omega}} \frac{\partial g_{k}}{\partial x},
- \frac{\widehat{G}}{\widehat{\omega}}\frac{\partial g_{k}}{\partial x} + \frac{\widehat{F}}{\widehat{\omega}} \frac{\partial g_{k}}{\partial y} \right), \quad k \in \{1, \cdots, n\}.
\]
Here, $ds_{\mathbf{\Phi}}^{2}=E dx^{2}+2F dx dy + Gdy^{2}$ denotes the induced metric and $\omega=\sqrt{EG-F^{2}}$. Likewise, we write $ds_{\mathbf{\widehat{\Phi}}}^{2}= \widehat{E}dx^{2}+2 \widehat{F}dx dy + \widehat{G}dy^{2}$ and $\widehat{\omega}=\sqrt{\widehat{E}\widehat{G}-{\widehat{F}}^{2}}$. \\
\textbf{(c2)} They share the area angle ${\Theta} \in \left(0, \frac{\pi}{2} \right]$. Each Jacobian determinant is preserved:
\[
\frac{ \partial \left( f_{i}, f_{j} \right)}{\partial (x, y)} = \frac{ \partial \left( g_{i}, g_{j} \right)}{\partial (x, y)}, \quad i, j \in \{1, \cdots, n\}.
\]
\textbf{(c3)} The angle duality holds
\[
\widehat{\omega} \omega = {\sin}^{2} \Theta >0.
\]
\textbf{(c4)} Two induced metrics are conformally equivalent. In fact,
$ds_{\mathbf{\Phi}}^{2}= \frac{\omega}{\widehat{\omega}}
ds_{\mathbf{\widehat{\Phi}}}^{2}$. \\
\textbf{(d)} Since two integrability conditions in (c1) are equivalent, we see that
the twin correspondence is involutive. The twin minimal surface
$\widehat{\widehat{\Sigma}}$ of the twin surface $\widehat{\Sigma}$ of a minimal surface ${\Sigma}$
is congruent to ${\Sigma}$ up to vertical translations.
\end{theorem}
\begin{remark}
Theorem \ref{twintwo2} extends Calabi's correspondence between minimal graphs in ${\mathbb{R}}^{3}$
and maximal graphs in ${\mathbb{L}}^{3}$. Here we explain the angle duality, appeared in (c3) in Theorem \ref{twintwo2}, in the
twin correspondence. First, in Calabi's correspondence, when $z=f(x,y)$ is a minimal graph in ${\mathbb{R}}^{3}$
and $z=g(x,y)$ is its twin maximal graph in ${\mathbb{L}}^{3}={\mathbb{R}}^{3}_{1}$, the angle duality reads
\[
\widehat{\omega} \omega = 1,
\]
where $\frac{1}{\omega}=\frac{1}{\sqrt{1+{f_{x}}^{2}+{f_{y}}^{2}}}=\left\langle N_{f}, {\partial}_{z} \right\rangle$
and $\frac{1}{\widehat{\omega}}=\frac{1}{\sqrt{1-{g_{x}}^{2}-{g_{y}}^{2}}}=\left\langle N_{g}, {\partial}_{z} \right\rangle$
are the induced angle functions of the minimal graph $z=f(x,y)$ and the maximal graph $z=g(x,y)$, respectively.
Here $N_{f}$ and $N_{g}$ denote their induced unit normal.
Theorem \ref{twintwo2} says that we have the twin correspondence between two dimensional minimal graphs having (not necessarily constant) positive area angle function ${\Theta} \in \left(0, \frac{\pi}{2} \right]$ in Euclidean space ${\mathbb{R}}^{n+2}$ and two dimensional maximal graphs having the same positive area angle function ${\Theta} \in \left(0, \frac{\pi}{2} \right]$ in pseudo-Euclidean space ${\mathbb{R}}^{n+2}_{n}$. In this case, the angle duality reads
\[
\widehat{\omega} \omega = {\sin}^{2} \Theta \leq 1,
\]
where ${\omega}$ and $\widehat{\omega}$ are defined in (c1) in Theorem \ref{twintwo2}. We notice that the right hand side
${\sin}^{2} \Theta$ in the general angle duality is not necessarily the constant $1$.
\end{remark}
\begin{remark}
As proved in \cite{Lee11a}, the twin correspondence in Theorem \ref{twintwo2} induces a natural duality for special Lagrangian graphs in ${\mathbb{R}}^{4}$ and ${\mathbb{R}}^{4}_{2}$.
\end{remark}
\subsection{{Minimal surface system and maximal surface system}} \label{MSSeu} 
The research on the minimal surface system is initiated in \cite{LO77, Oss70}.
Let ${\mathbf{e}}_{1}=(1,0, \cdots, 0)$, $\cdots$, ${\mathbf{e}}_{n+2}=(0,0, \cdots, 1)$ be the
standard basis in Euclidean space ${\mathbb{R}}^{n+2}$ endowed with the metric ${dx_{1}}^{2} +\cdots + {dx_{n+2}}^{2}$. The minimal surface system reads \cite{Oss86}:
\begin{prop} \label{MSS1}
The graph $\mathbf{\Phi}(x,y)=x {\mathbf{e}}_{1}+y{\mathbf{e}}_{2}+ f_{1}(x,y) {\mathbf{e}}_{3} + \cdots + f_{n}(x,y) {\mathbf{e}}_{n+2}$ of the height function $f=\left({f}_{1}, \cdots, {f}_{n}\right): \Omega \rightarrow {\mathbb{R}}^{n}$ becomes a minimal surface in ${\mathbb{R}}^{n+2}$ if and only if
the minimal surface system holds:
\[
G \frac{\partial^{2} f_{k}}{\partial x^{2}} - 2 F \frac{\partial^{2} f_{k}}{\partial x \partial y} + E \frac{\partial^{2} f_{k}}{\partial y^{2}}=0, \quad k \in \{1, \cdots, n\}.
\]
Every ${\mathcal{C}}^{2}$ solution of the minimal
surface system is real analytic. Here, we set
\[
\begin{cases}
E=1 + \left(
\frac{\partial f_{1}}{\partial x} \right)^{2} +\cdots + \left(
\frac{\partial f_{n}}{\partial x} \right)^{2}, \\
F= \quad \quad \frac{\partial f_{1}}{\partial x} \frac{\partial f_{1}}{\partial y}
+ \cdots + \frac{\partial f_{n}}{\partial x} \frac{\partial
f_{n}}{\partial y}, \\
G= 1 + \left( \frac{\partial
f_{1}}{\partial y} \right)^{2} + \cdots + \left( \frac{\partial
f_{n}}{\partial y} \right)^{2}.
\end{cases}
\]
The patch $\mathbf{\Phi}$ induces the metric $ds_{\mathbf{\Phi}}^{2}=E dx^{2}+2F dx dy +Gdy^{2}$.
\end{prop}
\begin{definition} Let $\omega=\sqrt{EG-F^{2}}>0$ denote the area element. We call
$\frac{1}{\omega}$ the angle function of the minimal graph $\mathbf{\Phi}$.
\end{definition}
\begin{prop} \label{MSS2} Let $\left({\alpha}_{k}, {\beta}_{k} \right)=\left(\frac{\partial f_{k}}{\partial x}, \frac{\partial f_{k}}{\partial y} \right)$, $k \in \{1, \cdots, n\}$. \\
\textbf{(a)} We can re-write the minimal surface system in zero divergence form:
\[
\frac{\partial }{\partial x} \left( \frac{G}{\omega} {\alpha}_{k} - \frac{F}{\omega} {\beta}_{k} \right)+
\frac{\partial }{\partial y} \left( \frac{E}{\omega} {\beta}_{k} - \frac{F}{\omega} {\alpha}_{k} \right)=0, \quad k \in \{1, \cdots, n\}.
\]
\textbf{(b)} We also have
\[
\frac{\partial }{\partial x} \left( \frac{G}{\omega} \right) = \frac{\partial }{\partial y} \left( \frac{F}{\omega} \right) \quad \text{and} \quad \frac{\partial }{\partial x} \left( \frac{F}{\omega} \right) = \frac{\partial }{\partial y} \left( \frac{E}{\omega} \right).
\]
\end{prop}
%
%
We now consider ${\mathbf{e}}_{1}=(1,0, \cdots, 0)$, $\cdots$, ${\mathbf{e}}_{n+2}=(0,0, \cdots, 1)$ as the
standard basis in pseudo-Euclidean space ${\mathbb{R}}^{n+2}_{n}$ equipped with the metric ${dx_{1}}^{2} + {dx_{2}}^{2} - {dx_{3}}^{2} -\cdots - {dx_{n+2}}^{2}$. The proofs of the below results are similar to the minimal surface system.
\begin{prop} \label{MSS3}
The spacelike graph $\mathbf{\widehat{\Phi}}(x,y)=x {\mathbf{e}}_{1}+y{\mathbf{e}}_{2}+ g_{1}(x,y) {\mathbf{e}}_{3} + \cdots + g_{n}(x,y) {\mathbf{e}}_{n+2}$ of the height function $g=\left({g}_{1}, \cdots, {g}_{n}\right): \Omega \rightarrow {\mathbb{R}}^{n}$ becomes a maximal surface in ${\mathbb{R}}^{n+2}_{n}$ if and only if $g$ satisfies the maximal surface
system:
\[
\widehat{G} \frac{\partial^{2} g_{k}}{\partial x^{2}} - 2 \widehat{F} \frac{\partial^{2} g_{k}}{\partial x \partial y} +
\widehat{E} \frac{\partial^{2} g_{k}}{\partial y^{2}}=0, \quad k \in \{1, \cdots, n\},
\]
or equivalently,
\[
\frac{\partial }{\partial x} \left( \frac{\widehat{G}}{\widehat{\omega}} {\widehat{\alpha}}_{k} - \frac{\widehat{F}}{\widehat{\omega}} {\widehat{\beta}}_{k} \right)+
\frac{\partial }{\partial y} \left( \frac{\widehat{E}}{\widehat{\omega}} {\widehat{\beta}}_{k} - \frac{\widehat{F}}{\widehat{\omega}} {\widehat{\alpha}}_{k} \right)=0, \quad \left({\widehat{\alpha}}_{k}, {\widehat{\beta}}_{k} \right)=\left(\frac{\partial g_{k}}{\partial x}, \frac{\partial g_{k}}{\partial y} \right).
\]
Here, we write
\[
\begin{cases}
\widehat{E}=1- \left( \frac{\partial g_{1}}{\partial x} \right)^{2} -\cdots - \left( \frac{\partial g_{n}}{\partial x} \right)^{2}, \\
\widehat{F}= \;\;\; - \frac{\partial g_{1}}{\partial x} \frac{\partial g_{1}}{\partial y} - \cdots - \frac{\partial g_{n}}{\partial x} \frac{\partial g_{n}}{\partial y}, \\
\widehat{G}= 1 - \left( \frac{\partial g_{1}}{\partial y} \right)^{2} - \cdots - \left( \frac{\partial g_{n}}{\partial y} \right)^{2}.
\end{cases}
\]
We assume the spacelike condition $\widehat{E}\widehat{G}-{\widehat{F}}^{2}>0$. The patch $\mathbf{\widehat{\Phi}}$ induces the Riemannian metric $ds_{\mathbf{\widehat{\Phi}}}^{2}= \widehat{E}dx^{2}+2 \widehat{F}dx dy + \widehat{G}dy^{2}$. We set $\widehat{\omega}=\sqrt{\widehat{E}\widehat{G}-{\widehat{F}}^{2}}$. We call $\frac{1}{\widehat{\omega}}$ the angle function of the maximal graph $\mathbf{\widehat{\Phi}}$ in ${\mathbb{R}}^{n+2}_{n}$.
Also, we obtain two identities
\[
\frac{\partial }{\partial x} \left( \frac{\widehat{G}}{\widehat{\omega}} \right) = \frac{\partial }{\partial y} \left( \frac{\widehat{F}}{\widehat{\omega}} \right)
\quad \text{and} \quad \frac{\partial }{\partial x} \left( \frac{\widehat{F}}{\widehat{\omega}} \right) = \frac{\partial }{\partial y} \left( \frac{\widehat{E}}{\widehat{\omega}} \right).
\]
\end{prop}
\subsection{Proof of twin correspondence II} \label{SecTwin}
We notice that Theorem \ref{twintwo2} is an extended version of Theorem \ref{twintwo}.
\begin{proof}[Proof of Theorem \ref{twintwo2}] \label{twinproof}
We deduce (a) and (c). Working backwards gives (b). \\
Let $(x,y) \in \Omega \mapsto \mathbf{\Phi}(x,y)=x {\mathbf{e}}_{1}+y{\mathbf{e}}_{2}+ f_{1}(x,y) {\mathbf{e}}_{3} + \cdots + f_{n}(x,y) {\mathbf{e}}_{n+2}$ be a minimal surface patch in ${\mathbb{R}}^{n+2}$. Write $ds_{\mathbf{\Phi}}^{2}=E dx^{2}+2F dx dy + Gdy^{2}$ and $\omega=\sqrt{EG-F^{2}}>0$. We further assume the positive area angle condition
\[
1 > \cos \Theta = \Vert \mathcal{J} \Vert := \sqrt{ \sum_{1 \leq i<j \leq n} {{\mathcal{J}}_{i,j}}^{2}}, \quad {\mathcal{J}}_{i,j}:= \frac{ \partial \left( f_{i}, f_{j} \right)}{\partial (x, y)}.
\]
After setting $\left({\alpha}_{k}, {\beta}_{k} \right)=\left(\frac{\partial f_{k}}{\partial x}, \frac{\partial f_{k}}{\partial y} \right)$, $k \in \{1, \cdots, n\}$, we obtain
\[
E= 1 + \sum_{k=1}^{n} {{\alpha}_{k}}^{2}, \; F= \sum_{k=1}^{n} {\alpha}_{k} {\beta}_{k}, \; G= 1+ \sum_{k=1}^{n} {{\beta}_{k}}^{2}, \; {\mathcal{J}}_{i,j}= {\alpha}_{i} {\beta}_{j} - {\alpha}_{j} {\beta}_{i}.
\]
Then, Lagrange's identity gives
\[
{\omega}^{2}= EG-F^{2} = 1 + \sum_{k=1}^{n} {{\alpha}_{k}}^{2} + \sum_{k=1}^{n} {{\beta}_{k}}^{2} + \sum_{1 \leq i < j \leq n} {{\mathcal{J}}_{i,j}}^{2}
\]
and thus
\[
{\Vert \mathcal{J} \Vert}^{2} =\sum_{1 \leq i < j \leq n} {{\mathcal{J}}_{i,j}}^{2} = {\omega}^{2} - 1 - (E-1) -(G-1)= {\omega}^{2} + 1 - E -G.
\]
Since the minimal surface system reads
\[
\frac{\partial }{\partial x} \left( \frac{G}{\omega} {\alpha}_{k} - \frac{F}{\omega} {\beta}_{k} \right)+
\frac{\partial }{\partial y} \left( \frac{E}{\omega} {\beta}_{k} - \frac{F}{\omega} {\alpha}_{k} \right)=0, \quad k \in \{1, \cdots, n\},
\]
and since $\Omega$ is simply connected, Poincar\'{e}'s Lemma
guarantees the existence of functions $g_{1}$, $\cdots$, $g_{n}:
\Omega \rightarrow {\mathbb{R}}$ satisfying the following integrability condition
\[
\left( \frac{\partial g_{k}}{\partial x}, \frac{\partial g_{k}}{\partial y} \right)
= \left( - \frac{E}{\omega} {\beta}_{k} + \frac{F}{\omega} {\alpha}_{k},
\frac{G}{\omega} {\alpha}_{k} - \frac{F}{\omega} {\beta}_{k} \right), \quad k \in \{1, \cdots, n\}.
\]
Our aim is to show that $(x,y) \mapsto \mathbf{\widehat{\Phi}}(x,y)=x {\mathbf{e}}_{1}+y{\mathbf{e}}_{2}+ g_{1}(x,y) {\mathbf{e}}_{3} + \cdots + g_{n}(x,y) {\mathbf{e}}_{n+2}$ becomes a maximal surface patch in ${\mathbb{R}}^{n+2}_{n}$. Write $ds_{\mathbf{\widehat{\Phi}}}^{2}= \widehat{E}dx^{2}+2 \widehat{F}dx dy + \widehat{G}dy^{2}$. Then, we compute
\[
\widehat{E}= 1 - \sum_{k=1}^{n} {{\widehat{\alpha}}_{k}}^{2}, \; \widehat{F}= - \sum_{k=1}^{n} {\widehat{\alpha}}_{k} {\widehat{\beta}}_{k}, \; \widehat{G}= 1-\sum_{k=1}^{n} {{\widehat{\beta}}_{k}}^{2}, \quad \left({\widehat{\alpha}}_{k}, {\widehat{\beta}}_{k} \right)=\left(\frac{\partial g_{k}}{\partial x}, \frac{\partial g_{k}}{\partial y} \right).
\]
The first part of (c1) is now finished because the above condition reads
\[
\left( {\widehat{\alpha}}_{k} , {\widehat{\beta}}_{k} \right)
= \left( - \frac{E}{\omega} {\beta}_{k} + \frac{F}{\omega} {\alpha}_{k},
\frac{G}{\omega} {\alpha}_{k} - \frac{F}{\omega} {\beta}_{k} \right).
\]
\textbf{Claim A.} We show the spacelike condition $\widehat{E}\widehat{G}-{\widehat{F}}^{2}>0$. \\
\textbf{Step A1.} We compute Jacobian determinants. The above twin relation reads
\[
\begin{pmatrix} {\widehat{\beta}}_{k} \\ -{\widehat{\alpha}}_{k} \end{pmatrix}
= \begin{pmatrix}
\frac{G}{\omega} & -\frac{F}{\omega} \\
-\frac{F}{\omega} & \frac{E}{\omega}
\end{pmatrix}
\begin{pmatrix} {{\alpha}}_{k} \\ {{\beta}}_{k} \end{pmatrix}
\quad \text{or} \quad
\begin{pmatrix} {{\alpha}}_{k} \\ {{\beta}}_{k} \end{pmatrix} = \begin{pmatrix}
\frac{E}{\omega} & \frac{F}{\omega} \\
\frac{F}{\omega} & \frac{G}{\omega}
\end{pmatrix}
\begin{pmatrix} {\widehat{\beta}}_{k} \\ -{\widehat{\alpha}}_{k} \end{pmatrix}.
\]
Equating the determinant of the both sides in
\[
\begin{pmatrix} {{\alpha}}_{i} & {{\alpha}}_{j} \\ {{\beta}}_{i} & {{\beta}}_{j} \end{pmatrix} = \begin{pmatrix}
\frac{E}{\omega} & \frac{F}{\omega} \\
\frac{F}{\omega} & \frac{G}{\omega}
\end{pmatrix}
\begin{pmatrix} {\widehat{\beta}}_{i} & {\widehat{\beta}}_{j} \\ -{\widehat{\alpha}}_{i} & -{\widehat{\alpha}}_{j} \end{pmatrix}
\]
yields the equality ${\alpha}_{i} {\beta}_{j} - {\alpha}_{j} {\beta}_{i}= {\widehat{\alpha}}_{i} {\widehat{\beta}}_{j} - {\widehat{\alpha}}_{j} {\widehat{\beta}}_{i}$, which gives the assertion (c2):
\[
{\mathcal{J}}_{i,j} := \frac{ \partial \left( f_{i}, f_{j} \right)}{\partial (x, y)}={\alpha}_{i} {\beta}_{j} - {\alpha}_{j} {\beta}_{i}= {\widehat{\alpha}}_{i} {\widehat{\beta}}_{j} - {\widehat{\alpha}}_{j} {\widehat{\beta}}_{i} = \frac{ \partial \left( g_{i}, g_{j} \right)}{\partial (x, y)}, \quad i, j \in \{1, \cdots, n\}.
\]
\textbf{Step A2.} Using the twin relation and the definition ${\omega}^{2}= EG-F^{2}$, we obtain
\begin{eqnarray*}
&& \sum_{k=1}^{n} \left[ {{\widehat{\alpha}}_{k}}^{2} + {{\widehat{\beta}}_{k}}^{2} \right] \\
& = & \sum_{k=1}^{n} \left[ { \left( - \frac{E}{\omega} {\beta}_{k} + \frac{F}{\omega} {\alpha}_{k} \right) }^{2} + { \left( \frac{G}{\omega} {\alpha}_{k} - \frac{F}{\omega} {\beta}_{k} \right) }^{2} \right] \\
& = & \frac{1}{{\omega}^{2}} \left[
\left( F^2 +G^2 \right) \sum_{k=1}^{n} {{\alpha}_{k}}^{2} + \left( F^2 +E^2 \right) \sum_{k=1}^{n} {{\beta}_{k}}^{2}
- 2 (E+G)F \sum_{k=1}^{n} {\alpha}_{k} {\beta}_{k} \right] \\
& = & \frac{1}{{\omega}^{2}} \left[ \left( F^2 +G^2 \right) (E-1) + \left( F^2 +E^2 \right) (G-1)
- 2 (E+G) F^{2} \right] \\
& = &\frac{ (E+G+2) \left(EG-F^2\right) - (E+G)^{2} }{{\omega}^{2}}.
\end{eqnarray*}
\textbf{Step A3.} We here deduce the identity
$\widehat{E}\widehat{G}-{\widehat{F}}^{2} = \frac{\left( 1 - {\Vert \mathcal{J} \Vert}^{2} \right)^{2} }{{\omega}^{2}}$.
Then, the desired estimation $\widehat{E}\widehat{G}-{\widehat{F}}^{2}>0$ immediately follows from our assumption $1 > \Vert \mathcal{J} \Vert$. We recall that the Jacobian determinant equality ${\mathcal{J}}_{i,j} = {\widehat{\alpha}}_{i} {\widehat{\beta}}_{j} - {\widehat{\alpha}}_{j} {\widehat{\beta}}_{i}$ holds. Using Lagrange's identity again, we deduce
\begin{eqnarray*}
\widehat{E}\widehat{G}-{\widehat{F}}^{2} &=& \left( 1 - \sum_{k=1}^{n} {{\widehat{\alpha}}_{k}}^{2}
\right)\left(1-\sum_{k=1}^{n} {{\widehat{\beta}}_{k}}^{2} \right) - \left(- \sum_{k=1}^{n} {\widehat{\alpha}}_{k} {\widehat{\beta}}_{k} \right)^{2} \\
&=& 1- \sum_{k=1}^{n} \left[ {{\widehat{\alpha}}_{k}}^{2} + {{\widehat{\beta}}_{k}}^{2} \right] +
\left( \sum_{k=1}^{n} {{\widehat{\alpha}}_{k}}^{2}
\right)\left( \sum_{k=1}^{n} {{\widehat{\beta}}_{k}}^{2} \right) - \left( \sum_{k=1}^{n} {\widehat{\alpha}}_{k} {\widehat{\beta}}_{k} \right)^{2} \\
&=& 1- \sum_{k=1}^{n} \left[ {{\widehat{\alpha}}_{k}}^{2} + {{\widehat{\beta}}_{k}}^{2} \right] +
\sum_{1 \leq i < j \leq n} {{\mathcal{J}}_{i,j}}^{2} \\
&=& 1- \frac{ (E+G+2) {\omega}^{2} - (E+G)^{2} }{{\omega}^{2}} + \left( {\omega}^{2} + 1 - E -G \right) \\
&=& \frac{ \left( E+G - {\omega}^{2} \right)^{2} }{{\omega}^{2}} \\
&=& \frac{ \left( 1 - {\Vert \mathcal{J} \Vert}^{2} \right)^{2} }{{\omega}^{2}}.
\end{eqnarray*}
Since $1 > \cos \Theta = \Vert \mathcal{J} \Vert$, we deduce $\widehat{E}\widehat{G}-{\widehat{F}}^{2}>0$.
Write $\widehat{\omega}:=\sqrt{\widehat{E}\widehat{G}-{\widehat{F}}^{2}}>0$. Then, the above equalities give
the angle duality in (c3):
\[
\widehat{\omega}=\sqrt{\widehat{E}\widehat{G}-{\widehat{F}}^{2}}=\frac{E+G-{\omega}^{2}}{\omega}=\frac{ 1 - {\Vert \mathcal{J} \Vert}^{2} }{\omega}=\frac{ {\sin}^{2} \Theta }{\omega}>0.
\]
\textbf{Claim B.} We verify that $\mathbf{\widehat{\Phi}}(x,y)=x {\mathbf{e}}_{1}+y{\mathbf{e}}_{2}+ g_{1}(x,y) {\mathbf{e}}_{3} + \cdots + g_{n}(x,y) {\mathbf{e}}_{n+2}$ is a maximal graph in ${\mathbb{R}}^{n+2}_{n}$. \\
\textbf{Step B1.} We check (c4), which implies that the metric $ds_{\mathbf{\widehat{\Phi}}}^{2}= \frac{\widehat{\omega}}{\omega} ds_{\mathbf{\Phi}}^{2}$ is positive definite. We need to prove the three identities $\frac{E}{\omega}=\frac{\widehat{E}}{\widehat{\omega}}$,
$\frac{F}{\omega}=\frac{\widehat{F}}{\widehat{\omega}}$ and $\frac{G}{\omega}=\frac{\widehat{G}}{\widehat{\omega}}$.
The twin relation and the definition ${\omega}^{2}= EG-F^{2}$ yield
\begin{eqnarray*}
\sum_{k=1}^{n} {{\widehat{\alpha}}_{k}}^{2}
& = & \sum_{k=1}^{n} { \left( - \frac{E}{\omega} {\beta}_{k} + \frac{F}{\omega} {\alpha}_{k} \right) }^{2} \\
& = & \frac{E^{2}}{{\omega}^{2}} \sum_{k=1}^{n} {{\beta}_{k}}^{2}
+ \frac{F^{2}}{{\omega}^{2}} \sum_{k=1}^{n} {{\alpha}_{k}}^{2}
-2 \frac{EF}{{\omega}^{2}} \sum_{k=1}^{n} {\alpha}_{k} {\beta}_{k} \\
&=& \frac{E^{2} (G-1) + F^{2} (E-1) -2EF^{2}} {{\omega}^{2}} \\
&=& \frac{E {\omega}^{2} - E^2 -F^2 } {{\omega}^{2}}.
\end{eqnarray*}
We then use the angle duality
$\widehat{\omega}=\frac{ 1 - {\Vert \mathcal{J} \Vert}^{2} }{\omega}=\frac{E+G-{\omega}^{2}}{\omega}$
to deduce
\[
\widehat{E} - \frac{\widehat{\omega}}{\omega} {E} = 1- \sum_{k=1}^{n} {{\widehat{\alpha}}_{k}}^{2} -
\frac{\widehat{\omega}}{\omega} E = 1- \frac{E {\omega}^{2} - E^2 -F^2 } {{\omega}^{2}} -
\frac{E+G-{\omega}^{2}}{{\omega}^{2}} E=0.
\]
The remaining two identities in (c4) is proved similarly. Thus, (c4) is proved. \\
\textbf{Step B2.} We show that the height function $g$ satisfies the maximal surface system. We
use the three identities in Step B1 to rewrite the twin relation
as
\[
\begin{pmatrix} \frac{\partial f_{k}}{\partial x} \\ \frac{\partial f_{k}}{\partial y} \end{pmatrix}
=\begin{pmatrix} {{\alpha}}_{k} \\ {{\beta}}_{k} \end{pmatrix} = \begin{pmatrix}
\frac{E}{\omega} & \frac{F}{\omega} \\
\frac{F}{\omega} & \frac{G}{\omega}
\end{pmatrix}
\begin{pmatrix} {\widehat{\beta}}_{k} \\ -{\widehat{\alpha}}_{k} \end{pmatrix}
= \begin{pmatrix}
\frac{\widehat{E}}{\widehat{\omega}} & \frac{\widehat{F}}{\widehat{\omega}} \\
\frac{\widehat{F}}{\widehat{\omega}} & \frac{\widehat{G}}{\widehat{\omega}}
\end{pmatrix}
\begin{pmatrix} {\widehat{\beta}}_{k} \\ -{\widehat{\alpha}}_{k} \end{pmatrix}
= \begin{pmatrix}
\frac{\widehat{E}}{\widehat{\omega}} {\widehat{\beta}}_{k} - \frac{\widehat{F}}{\widehat{\omega}} {\widehat{\alpha}}_{k} \\
-\frac{\widehat{F}}{\widehat{\omega}} {\widehat{\alpha}}_{k} + \frac{\widehat{G}}{\widehat{\omega}} {\widehat{\beta}}_{k}
\end{pmatrix},
\]
which is the second part of (c1). It therefore follows that, for all $k \in \{1, \cdots, n\}$,
\[
\frac{\partial }{\partial x} \left( \frac{\widehat{G}}{\widehat{\omega}} {\widehat{\alpha}}_{k} - \frac{\widehat{F}}{\widehat{\omega}} {\widehat{\beta}}_{k} \right)+
\frac{\partial }{\partial y} \left( \frac{\widehat{E}}{\widehat{\omega}} {\widehat{\beta}}_{k} - \frac{\widehat{F}}{\widehat{\omega}} {\widehat{\alpha}}_{k} \right)= \frac{\partial }{\partial x} \left(-\frac{\partial f_{k}}{\partial y} \right) + \frac{\partial }{\partial y} \left( \frac{\partial f_{k}}{\partial x} \right)= 0,
\]
which is the maximal surface system. This completes the proof of (a).
\end{proof}

\begin{remark}The key point of the proof of Theorem \ref{twintwo2} is to show that the system of integrability conditions
\[
\left( \frac{\partial g_{k}}{\partial x}, \frac{\partial g_{k}}{\partial y} \right)
= \left( - \frac{E}{\omega} \frac{\partial f_{k}}{\partial y} + \frac{F}{\omega} \frac{\partial f_{k}}{\partial x},
\frac{G}{\omega}\frac{\partial f_{k}}{\partial x} - \frac{F}{\omega} \frac{\partial g_{k}}{\partial y} \right), \quad k \in \{1, \cdots, n\}
\]
is equivalent to the system of integrability conditions
\[
\left( \frac{\partial f_{k}}{\partial x}, \frac{\partial f_{k}}{\partial y} \right)
= \left( \frac{\widehat{E}}{\widehat{\omega}} \frac{\partial g_{k}}{\partial y} - \frac{\widehat{F}}{\widehat{\omega}} \frac{\partial g_{k}}{\partial x},
- \frac{\widehat{G}}{\widehat{\omega}}\frac{\partial g_{k}}{\partial x} + \frac{\widehat{F}}{\widehat{\omega}} \frac{\partial g_{k}}{\partial y} \right), \quad k \in \{1, \cdots, n\}.
\]
The proof of this fact is not immediate. In Theorem \ref{twintwo2}, we called these two \textit{equivalent} integrability conditions \textit{twin relations}. It is the reason why the proof of the higher codimension extension of Calabi's duality is not simple. The trick is to show that the induced metrics of twin surfaces
are conformally equivalent. This is Step B1 in the proof of Theorem \ref{twintwo2}.
\end{remark}

\subsection{Reading Weierstrass formulas of twin surfaces} \label{SecHOL}
We mentioned that the Calabi correspondence via Poincar\'{e}'s Lemma and the L\'{o}pez--L\'{o}pez--Souam correspondence via the Weierstrass data are the same \cite{Lee10a}. We extend this result for the twin correspondence between two dimensional minimal graphs having positive area angles in ${\mathbb{R}}^{n+2}$ and two dimensional maximal graphs having positive area angles in ${\mathbb{R}}^{n+2}_{n}$. The main idea is to construct simultaneous conformal coordinates for twin graphs.
\begin{theorem}[\cite{Lee11a}] \label{SCC}
Let $\mathbf{\Phi}(x,y)$ be the minimal graph in ${\mathbb{R}}^{n+2}$ of an area-decreasing map $f=\left({f}_{1}, \cdots, {f}_{n}\right): \Omega \rightarrow {\mathbb{R}}^{n}$ with the (not necessarily constant) area angle function ${\Theta} \in \left(0, \frac{\pi}{2} \right]$ and $\mathbf{\widehat{\Phi}}(x,y)$ its twin maximal graph in ${\mathbb{R}}^{n+2}_{n}$ of an area-decreasing map $g=\left({g}_{1}, \cdots, {g}_{n}\right): \Omega \rightarrow {\mathbb{R}}^{n}$ with
the same area angle function ${\Theta} \in \left(0, \frac{\pi}{2} \right]$. \\
\textbf{(a)} There exists a simultaneous conformal coordinate $\xi={\xi}_{1}+i{\xi}_{2}$ for the minimal graph $\mathbf{\Phi}(x,y)$ in ${\mathbb{R}}^{n+2}$ and its twin maximal graph $\mathbf{\widehat{\Phi}}(x,y)$ in ${\mathbb{R}}^{n+2}_{n}$. \\
\textbf{(b)} Let $\Psi: (x,y) \mapsto \left( {\xi}_{1}, {\xi}_{2} \right)$ be the coordinate transformation in (a). The conformal harmonic immersion $\mathbf{\Phi} \circ {\Psi}^{-1}$ induces
the holomorphic null curve
\[
2\frac{\partial}{\partial {\xi}} \mathbf{\Phi} \circ {\Psi}^{-1}=\left({\phi}_{1}, {\phi}_{2}, {\phi}_{3}, \cdots, {\phi}_{n+2} \right).
\]
The conformal harmonic immersion $\mathbf{\widehat{\Phi}} \circ {\Psi}^{-1}$ induces the holomorphic
null curve
\[
2\frac{\partial}{\partial {\xi}} \mathbf{\widehat{\Phi}} \circ {\Psi}^{-1}= \left({\widehat{\phi}}_{1}, {\widehat{\phi}}_{2}, {\widehat{\phi}}_{3}, \cdots, {\widehat{\phi}}_{n+2} \right).
\]
We have the twin relation
${\phi}_{1}={\widehat{\phi}}_{1}$, ${\phi}_{2}={\widehat{\phi}}_{2}$, ${\phi}_{k+2}=-i{\widehat{\phi}}_{k+2}, \; k \in \{1, \cdots, n\}$.
\end{theorem}
\begin{proof} We begin with the two identities in Proposition \ref{MSS2}:
\[
\frac{\partial }{\partial x} \left( \frac{F}{\omega} \right) = \frac{\partial }{\partial y} \left( \frac{E}{\omega} \right) \quad \text{and} \quad
\frac{\partial }{\partial x} \left( \frac{G}{\omega} \right) = \frac{\partial }{\partial y} \left( \frac{F}{\omega} \right).
\]
Here, $ds_{\mathbf{\Phi}}^{2}=E dx^{2}+2F dx dy + Gdy^{2}$ denotes the induced metric of the minimal graph $\mathbf{\Phi}(x,y)$ in ${\mathbb{R}}^{n+2}$. Since $\Omega$ is simply connected, Poincar\'{e}'s Lemma guarantees the existence of functions $N, M:\Omega \rightarrow \mathbb{R}$ satisfying the equalities
\[
M_{x} = \frac{E}{\omega}, \quad M_{y} = \frac{F}{\omega}, \quad N_{x} = \frac{F}{\omega}, \quad N_{y} = \frac{G}{\omega}.
\]
We then are able to introduce the coordinate transformation
\[
\Psi: (x,y) \mapsto \left( {\xi}_{1}, {\xi}_{2} \right):=\left(x+M(x,y),
y+N(x,y) \right)
\]
such that
\[
J_{\Psi}= \frac{ \partial \left( {\xi}_{1}, {\xi}_{2} \right)}{\partial (x, y)} = \det
\begin{pmatrix} 1+ \frac{E}{\omega} & \frac{F}{\omega} \\ \frac{F}{\omega} & 1+ \frac{G}{\omega}
\end{pmatrix} = 2+ \frac{E+G}{\omega}>2.
\]
Since $ J_{\Psi}=\frac{ \partial \left( {\xi}_{1}, {\xi}_{2} \right)}{\partial (x, y)}>0$, we have the existence of
the local inverse $\left( {\xi}_{1}, {\xi}_{2} \right) \mapsto (x,y)$.
Now, using the Chain Rule, we obtain the conformal metrics:
\[
ds_{\mathbf{\Phi}}^{2}=\frac{\omega}{J_{\Psi}} \left( d{{\xi}_{1}}^{2} + d{{\xi}_{2}}^{2} \right) \quad
\text{and} \quad ds_{\mathbf{\widehat{\Phi}}}^{2}= \frac{\widehat{\omega}}{J_{\Psi}} \left( d{{\xi}_{1}}^{2} + d{{\xi}_{2}}^{2} \right) = \frac{\widehat{\omega}}{\omega} ds_{\mathbf{\Phi}}^{2}.
\]
Proof of (a) is finished. We prove (b). Consider two conformal immersions
\[
\begin{cases}
\mathbf{\Phi} \circ {\Psi}^{-1} =
x \left( {\xi}_{1}, {\xi}_{2} \right) {\mathbf{e}}_{1}+y \left( {\xi}_{1}, {\xi}_{2} \right){\mathbf{e}}_{2}+
\sum_{k=1}^{n} f_{k}(x \left( {\xi}_{1}, {\xi}_{2} \right),y \left( {\xi}_{1}, {\xi}_{2} \right)) {\mathbf{e}}_{k+2},\\
\mathbf{\widehat{\Phi}} \circ {\Psi}^{-1} =
x \left( {\xi}_{1}, {\xi}_{2} \right) {\mathbf{e}}_{1}+y \left( {\xi}_{1}, {\xi}_{2} \right){\mathbf{e}}_{2}+
\sum_{k=1}^{n} g_{k}(x \left( {\xi}_{1}, {\xi}_{2} \right),y \left( {\xi}_{1}, {\xi}_{2} \right)) {\mathbf{e}}_{k+2}.
\end{cases}
\]
By the definition of the induced holomorphic curves, obviously, we obtain ${\widehat{\phi}}_{1}={\phi}_{1}$ and ${\widehat{\phi}}_{2}={\phi}_{2}$. To deduce ${\phi}_{k+2}=-i{\widehat{\phi}}_{k+2}$, $k \in \{1, \cdots, n\}$, we show that
\[
\left( \frac{\partial f_{k}}{\partial {\xi}_{1}}, \frac{\partial f_{k}}{\partial {\xi}_{2}} \right)
= \left( \frac{\partial g_{k}}{\partial {\xi}_{2}}, - \frac{\partial g_{k}}{\partial {\xi}_{1}} \right).
\]
We only check the second components. The twin relation in Theorem \ref{twintwo2} reads
\[
\left( {\widehat{\alpha}}_{k} , {\widehat{\beta}}_{k} \right)
= \left( - \frac{E}{\omega} {\beta}_{k} + \frac{F}{\omega} {\alpha}_{k},
\frac{G}{\omega} {\alpha}_{k} - \frac{F}{\omega} {\beta}_{k} \right),
\]
where $\left({\alpha}_{k}, {\beta}_{k}, {\widehat{\alpha}}_{k}, {\widehat{\beta}}_{k} \right)=\left(\frac{\partial f_{k}}{\partial x}, \frac{\partial f_{k}}{\partial y}, \frac{\partial g_{k}}{\partial x}, \frac{\partial g_{k}}{\partial y} \right)$.
We now prepare the equality
\[
\begin{pmatrix} \frac{\partial x}{\partial {\xi}_{1}} & \frac{\partial x}{\partial {\xi}_{2}} \\
\frac{\partial y}{\partial {\xi}_{1}} & \frac{\partial y}{\partial {\xi}_{2}} \end{pmatrix} =
{ \begin{pmatrix} \frac{\partial {\xi}_{1}}{\partial x} & \frac{\partial {\xi}_{1}}{\partial y} \\
\frac{\partial {\xi}_{2}}{\partial x} & \frac{\partial {\xi}_{2}}{\partial y} \end{pmatrix} }^{-1} = \frac{1}{J_{\Psi}}
\begin{pmatrix} 1+ \frac{G}{\omega} & - \frac{F}{\omega} \\ - \frac{F}{\omega} & 1+ \frac{E}{\omega}
\end{pmatrix}.
\]
We then use this together with the Chain Rule to deduce
\begin{eqnarray*}
&& \frac{\partial g_{k}}{\partial {\xi}_{1}} \\
&=&\frac{\partial x}{\partial {\xi}_{1}} \frac{\partial g_{k}}{\partial x} + \frac{\partial y}{\partial {\xi}_{1}} \frac{\partial g_{k}}{\partial y} \\
&=&\frac{\partial x}{\partial {\xi}_{1}} \left[ - \frac{E}{\omega} {\beta}_{k} + \frac{F}{\omega} {\alpha}_{k} \right] + \frac{\partial y}{\partial {\xi}_{1}} \left[ \frac{G}{\omega} {\alpha}_{k} - \frac{F}{\omega} {\beta}_{k} \right] \\
&=& \left[ \frac{\partial x}{\partial {\xi}_{1}} \frac{F}{\omega} + \frac{\partial y}{\partial {\xi}_{1}}
\frac{G}{\omega} \right] {\alpha}_{k} - \left[ \frac{\partial x}{\partial {\xi}_{1}} \frac{E}{\omega} + \frac{\partial y}{\partial {\xi}_{1}} \frac{F}{\omega} \right] {\beta}_{k} \\
&=& \frac{1}{J_{\Psi}} \left[ \left( 1+ \frac{G}{\omega} \right) \frac{F}{\omega} +
\left( - \frac{F}{\omega} \right) \frac{G}{\omega} \right] {\alpha}_{k} -
\frac{1}{J_{\Psi}} \left[ \left( 1+ \frac{G}{\omega} \right) \frac{E}{\omega} + \left( - \frac{F}{\omega} \right) \frac{F}{\omega} \right] {\beta}_{k} \\
&=& \frac{1}{J_{\Psi}} \frac{F}{\omega} {\alpha}_{k} - \frac{1}{J_{\Psi}} \left( 1+ \frac{E}{\omega} \right) {\beta}_{k} \\
&=& - \frac{\partial x}{\partial {\xi}_{2}} \frac{\partial f_{k}}{\partial x} - \frac{\partial y}{\partial {\xi}_{2}} \frac{\partial f_{k}}{\partial y} \\
&=& - \frac{\partial f_{k}}{\partial {\xi}_{2}}.
\end{eqnarray*}
\end{proof}
\section{Closing comments}
Two twin correspondences we constructed admit natural generalizations:
\begin{remark}[Timelike surfaces in Lorentzian Bianchi--Cartan--Vranceanu space]
We met the twin correspondence between graphs of constant mean curvature $H$ in
$\mathbb{E}^{3}(\kappa, \tau)$ and spacelike graphs of constant mean curvature $\tau$ in $\mathbb{L}^{3}(\kappa, H)$.
Similarly, we can construct the twin correspondence between timelike graphs of constant mean curvature $H$ in $\mathbb{L}^{3}(\kappa, \tau)$ and timelike graphs of constant mean curvature $\tau$ in $\mathbb{L}^{3}(\kappa, H)$. It extends
the classical self-duality on the moduli space of timelike graphs with zero mean curvature in Lorentz space $\mathbb{L}^{3}$.
\end{remark}
\begin{remark}[Extremal surfaces in  pseudo-Euclidean spaces]
We have the twin correspondence between two dimensional minimal graphs having positive area angles in Euclidean space ${\mathbb{R}}^{n+2}$ and two dimensional maximal graphs having the same positive area angles in pseudo-Euclidean space ${\mathbb{R}}^{n+2}_{n}$ with signature $(+,+,-,\cdots,-)$. Theorem \ref{twintwo2} and Theorem \ref{SCC} can be
 modified to the twin correspondence for spacelike (or timelike) graphs of zero mean curvature in other pseudo-Euclidean spaces with different signatures.
\end{remark}
We propose to implement Shiffman's program on other geometric partial differential equations:
\begin{remark}[Hamiltonian stationary Lagrangian graphs in ${\mathbb{R}}^{4}$ \cite{Sch05, Sch06, SW03}] A smooth Lagrangian surface $\Sigma$ in Euclidean space ${\mathbb{R}}^{4}$ equipped with standard symplectic structure is \textit{hamiltonian stationary} when its Lagrangian angle function ${\theta}_{\Sigma}$ is harmonic on $\Sigma$. The potential function $f$ of hamiltonian stationary graph $(x,y,{f}_{x}, {f}_{y})$ satisfies the zero divergence equation of bi-harmonic type:
\[
0=\frac{\partial}{\partial x} \left( {\triangle}_{ds^{2}} f_{x} \right) +
\frac{\partial}{\partial y} \left( {\triangle}_{ds^{2}} f_{y} \right).
\]
Here $ds^{2}$ and ${\triangle}_{ds^{2}}$ denote the induced metric of $\Sigma$ and the associated Laplace-Betrami
operator, respectively. Poincar\'{e}'s Lemma yields the existence of the potential function
$g$ satisfying that
\[
\left(g_{x},g_{y}\right) = \left(-{\triangle}_{ds^{2}} f_{y}, {\triangle}_{ds^{2}} f_{x}\right).
\]
It is then natural to ask what is the \textit{geometry} of the potential function $g$?
\end{remark}
\begin{remark}[Hessian zero equation as zero divergence equation] Hartman and Nirenberg \cite{HN59}, Massey \cite{Mas62}, and Stoker \cite{Sto61} independently showed that any complete flat surface in Euclidean space ${\mathbb{R}}^{3}$ is a right cylinder over a complete planar curve \cite{Gal09}. The height function $f$ of the flat graph $z=f(x,y)$ in ${\mathbb{R}}^{3}$ satisfies zero Gaussian curvature equation
\[
0= \frac{f_{xx}f_{yy}-{f_{xy}}^{2}}{\left(1+{f_{x}}^{2}+{f_{y}}^{2}\right)^{2}}.
\]
Lam \cite{Bra11, Lam10} used a remarkable formula for the scalar curvature of $n$ dimensional graphs in Euclidean space ${\mathbb{R}}^{n+1}$ as the divergence of an $n$ dimensional vector field to present a simple proof
of Riemannian positive mass inequality for graphs over the base space ${\mathbb{R}}^{n}$. His beautiful formula says that two dimensional Hessian zero equation can be written as the zero divergence equation
\[
0=\frac{\partial}{\partial x} \left( \frac{f_{yy}f_{x} - f_{xy}f_{y}}{2(1+{f_{x}}^{2}+{f_{y}}^{2})} \right)+
\frac{\partial}{\partial y} \left( \frac{f_{xx}f_{y} - f_{xy}f_{x}}{2(1+{f_{x}}^{2}+{f_{y}}^{2})} \right).
\]
Here, we assume that $f$ is class of ${\mathcal{C}}^{3}$. Then, applying Poincar\'{e}'s Lemma to Hessian zero equation
yields the existence of the potential function $g$ such that
\[
\left(g_{x},g_{y}\right) = \left(- \frac{f_{xx}f_{y} - f_{xy}f_{x}}{2(1+{f_{x}}^{2}+{f_{y}}^{2})},
\frac{f_{yy}f_{x} - f_{xy}f_{y}}{2(1+{f_{x}}^{2}+{f_{y}}^{2})} \right).
\]
What is the \textit{geometry} of the potential function $g$?
\end{remark}

\bigskip
\bigskip
\bigskip
\begin{flushleft}
{{\texttt{Department of Mathematical Sciences,}}} \\
{{\texttt{Seoul National University, Korea.}}}
\end{flushleft}
\bigskip
\begin{flushleft}
{{\texttt{ultrametric@gmail.com}} } 
\end{flushleft}
\end{document}